\documentclass[11pt,reqno]{amsart}
\usepackage{amssymb, amsmath, amsthm,amsrefs,marvosym,wasysym}
\usepackage{latexsym}
\usepackage{color}
\usepackage{exscale}
\usepackage{fullpage}
\usepackage{rotating}
\usepackage{bm, bbm, mathrsfs}
\usepackage{graphics, graphicx}
\newtheorem{theorem}{Theorem}[section]
\newtheorem{lemma}[theorem]{Lemma}
\newtheorem{proposition}[theorem]{Proposition}

\newtheorem{definition}{Definition}[section]

\newtheorem{remark}{Remark}[section]

\newtheorem{notation}{Notation}

\def\eq#1{(\ref{#1})}
\def\nn{\nonumber}
\def\({\left(\begin{array}{cccccc}}
\def\){\end{array}\right)}

\def\eq#1{(\ref{#1})}
\def\nn{\nonumber}

\def\({\left(\begin{array}{cccccc}}
\def\){\end{array}\right)}

\def\bes{\begin{eqnarray}}
\def\ees{\end{eqnarray}}

\newcommand{\beq}{\begin{equation}}
\newcommand{\eeq}{\end{equation}}
\newcommand{\bea}{\begin{eqnarray}}
\newcommand{\eea}{\end{eqnarray}}
\newcommand{\beann}{\begin{eqnarray*}}
\newcommand{\eeann}{\end{eqnarray*}}

\newcommand{\eps}{\ensuremath{\varepsilon}}

\newcommand{\RR}{\mathbb{R}}

\DeclareMathOperator{\sgn}{sgn}

\DeclareMathOperator{\grad}{grad}
\DeclareMathOperator{\dv}{div}

\numberwithin{equation}{section}
 
\begin{document}

\title{Amplitude blowup in compressible Euler flows without shock formation}

\begin{abstract}
	Recent works have demonstrated that continuous self-similar radial Euler flows can
	drive primary (non-differentiated) flow variables to infinity at the center of motion. 
	Among the variables that blow up at collapse is the pressure, and it is unsurprising that 
	this type of behavior can generate an outgoing shock wave. 
	
	In this work we prove that there is an alternative scenario in which an incoming, 
	continuous 3-d flow suffers blowup, including in pressure,
	and yet remains continuous beyond collapse. We verify that this behavior is possible
	even in cases where the fluid is everywhere moving toward the center of motion at time of collapse.
	The results underscore the subtlety of shock formation in multi-dimensional flow.


\end{abstract}

\author{Helge Kristian Jenssen }\address{H.~K.~Jenssen, Department of
Mathematics, Penn State University,
University Park, State College, PA 16802, USA ({\tt
jenssen@math.psu.edu}).}

\date{\today}
\maketitle

\tableofcontents

%%%%%%%%%%%%%%%%%%%%%%%%%%%%%%%%%%%%%%%%%%%%%%%%
%%%%%%%%%%%%%%%%%%%%%%%%%%%%%%%%%%%%%%%%%%%%%%%%
\section{Introduction}\label{intro}
%%%%%%%%%%%%%%%%%%%%%%%%%%%%%%%%%%%%%%%%%%%%
The non-isentropic Euler system describes the time evolution of 
a compressible fluid in the absence of viscosity and heat conduction:
\begin{align}
	\rho_t+\dv_{\bf x}(\rho \bf u)&=0 \label{mass_m_d_full_eul}\\
	(\rho{\bf  u})_t+\dv_{\bf x}[\rho {\bf  u}\otimes{\bf  u}]+\grad_{\bf  x} p&=0
	\label{mom_m_d_full_eul}\\
	(\rho E)_t+\dv_{\bf x}[(\rho E+p){\bf  u}]&=0.\label{energy_m_d_full_eul}
\end{align}
The independent variables are time $t$ and position ${\bf  x}\in\RR^n$ ($n=2,3$), and the 
dependent variables are density $\rho$, fluid velocity ${\bf  u}$, 
and specific internal energy $e$; the total energy density is $E=e+\textstyle\frac{1}{2}|{\bf  u}|^2$. 
We assume the fluid is an ideal gas with adiabatic index $\gamma>1$, so that the pressure 
$p$ is given by
\beq\label{pressure1}
	p(\rho,e)=(\gamma-1)\rho e,
\eeq
and also polytropic, i.e., the specific internal energy is proportional to 
the absolute temperature $\theta$ of the gas.
The sound speed $c$ is then given by
\beq\label{sound_speed}
	c=\sqrt{\textstyle\frac{\gamma p}{\rho}}=\sqrt{\gamma(\gamma-1)e}\propto\sqrt\theta.
\eeq
The unknowns $\rho$, $\bf u$, $p$, $c$, $\theta$ are referred to as primary (i.e., undifferentiated) 
flow variables.

In what follows we specialize to radial flows, i.e., solutions to 
\eq{mass_m_d_full_eul}-\eq{energy_m_d_full_eul} 
where the flow variables depend on position only through $r=|{\bf x}|$, 
and the velocity field is purely radial, i.e., 
\[{\bf u}(t,{\bf x})=u(t,r)\textstyle\frac{\bf x}{r}\]
Choosing $\rho$, $u$, $c$ as dependent variables, the Euler system 
\eq{mass_m_d_full_eul}-\eq{energy_m_d_full_eul} reduces to
\begin{align}
	\rho_t+u\rho_r+\rho(u_r+\textstyle\frac{(n-1)u}{r}) &= 0\label{m_eul}\\
	u_t+ uu_r +\textstyle\frac{1}{\gamma\rho}(\rho c^2)_r&= 0\label{mom_eul}\\
	c_t+uc_r+{\textstyle\frac{\gamma-1}{2}}c(u_r+\textstyle\frac{(n-1)u}{r})&=0\label{ener_eul}
\end{align}
where $r>0$, $\rho=\rho(t,r)$, $u=u(t,r)$, and $c=c(t,r)$. 
In continuous Euler flow, the specific entropy $S$, specified via Gibbs' relation 
$de=\theta dS-pd(\frac{1}{\rho})$, is transported along particle paths. For continuous 
radial flows we therefore have
\beq\label{spec_entr}
	S_t+uS_r=0.
\eeq

Pioneering works by Guderley \cite{gud} and Landau-Stanyukovich \cite{stan} in the 1940s initiated 
the study of {\em self-similar} radial Euler flows. Self-similarity provides a radical simplification,
reducing \eq{m_eul}-\eq{ener_eul} to a system of three non-linear ODEs.
A key observation is that the sub-system of ODEs for the similarity variables $V$ and $C$
corresponding to $u$ and $c$, yields a single, autonomous ODE relating $V$ and $C$ 
(cf.\ \eq{CV_ode} below). The upshot is that one can effectively analyze various types of
self-similar Euler flows by studying a planar phase portrait. An important motivation 
for studying this type of solutions in detail is that they exhibit different types of singular behaviors. 
The present work concerns one aspect of this.

While the study of self-similar Euler solutions has generated a 
considerable literature by now (see, e.g.\ references in \cite{jj}), it is only lately that 
rigorous results have been obtained. Indeed, the existence of the
converging-expanding shock flows originally studied by Guderley and Landau-Stanyukovich, 
has been established in a strict mathematical sense only very recently \cite{jls1}. 

Another scenario, describing the collapse of a spherical cavity, was analyzed
by Hunter \cite{hun_60}; see also \cite{bk}. Building on earlier works, Lazarus \cite{laz} provided a detailed,
and partly numerical, analysis of the Guderley and Hunter solutions.

A different type of focusing solutions - the setting of the present work - 
is provided by {\em continuous} flows in which a converging wave 
collapses and blows up at the center of motion. 
One motivation for considering this latter type of solutions is to
clarify the blowup mechanism. Specifically, that blowup is a pure focusing 
effect that can occur even in the presence of an
everywhere positive pressure field (in contrast to the flows considered 
by Guderley, Landau-Stanyukovich, and Hunter.) 
A first, rigorous construction of continuous blowup was done in 
\cite{jt2} for the simplified isothermal model: 
A continuous (but not $C^1$) wave converges toward the origin, collapses
(implodes), and reflects off an expanding shock wave.
A similar construction in \cite{jt3} (depending on numerically drawn phase portraits) 
indicated the same type of behavior for the isentropic model.

A far more challenging task was accomplished in \cite{mrrs1} where {\em smooth} ($C^\infty$)
self-similar blowup solutions were constructed (up to collapse) for the first time; see also \cite{bcg}. 
In turn, these were used to generate nearby blowup solutions of the compressible
Navier-Stokes system in \cite{mrrs2} - a fundamental result in fluid dynamics. Recently, 
\cite{ccsv} shows how non-radial perturbations 
of the smooth imploding Euler profiles from \cite{mrrs1} provide examples of vorticity 
blowup in compressible flow.

In a different direction, the work \cite{jj} considered self-similar Euler flows 
exhibiting another type of blowup behavior. These solutions remain locally bounded 
near the center of motion while several of the flow variables suffer gradient blowup
at time of collapse. Somewhat surprisingly,  \cite{jj} provides numerical 
evidence that such a scenario does not necessarily 
lead to shock formation in the ensuing flow, at least for sufficiently large 
values of the adiabatic index ($\gamma\gtrapprox10$). 
The main objective of the present work is to prove that this
phenomenon, i.e., absence of shock formation, can occur even in 
the more singular case of amplitude blowup (for any $\gamma>1$).

The challenging general problem of analyzing multi-d shock formation and propagation 
has been treated in a number of recent works, cf. Section 1.3.1.\ in \cite{ccsv}
for a review.
The present work illustrates one subtlety of shock formation, or rather lack thereof,
by focusing (pun intended) attention on non-generic, converging-diverging
symmetric flows.

%%%%%%%%%%%%%%%%%%%%%%%%%%%
\subsection{Main result and outline}
%%%%%%%%%%%%%%%%%%%%%%%%%%%
It is well-known that the radial Euler system \eq{m_eul}-\eq{mom_eul}-\eq{ener_eul} 
admits self-similar solutions, see \cites{laz,gud,cf,sed,stan,rj}.
We follow the setup in \cites{laz,cf} and posit 
\beq\label{sim_vars}
	x=\textstyle\frac{t}{r^\lambda},\qquad \rho(t,r)=r^\kappa R(x),\qquad 
	u(t,r)=-\frac{r^{1-\lambda}}{\lambda}\frac{V(x)}{x}, \qquad 
	c(t,r)=-\frac{r^{1-\lambda}}{\lambda}\frac{C(x)}{x},
\eeq
where the similarity parameters $\kappa$ and $\lambda$ are a priori free.
Our main findings are as follows:
%%%%%%%%%%%%%%%%%%%%%%%%%%%%
\begin{theorem}\label{thm}
	Consider the 3-dimensional compressible Euler system 
	\eq{mass_m_d_full_eul}-\eq{energy_m_d_full_eul} 
	for an ideal gas \eq{pressure1} with adiabatic index $\gamma>1$, and 
	let the similarity parameter $\kappa$ have the ``isentropic'' value 
	$\bar \kappa=-\textstyle\frac{2(\lambda-1)}{\gamma-1}$. 
	
	Then, for any $\lambda>1$ sufficiently close to $1$ (depending on $\gamma$), the Euler system
	\eq{mass_m_d_full_eul}-\eq{energy_m_d_full_eul} admits radially symmetric, 
	self-similar solutions of the form \eq{sim_vars} with the following properties:
	\begin{itemize}
		\item[(i)] They are defined on all of $\RR_t\times \RR^3_{\bf x}$ and continuous
		except at the single point $(t,{\bf x})=(0,0)$;
		\item[(ii)] At time $t=0$  the density, velocity, pressure, sound speed,
		and temperature all tend to infinity as the origin ${\bf x}=0$ is approached;
		\item[(iii)] The solutions contain locally finite amounts of mass, momentum, and energy
		at all times, and describe globally isentropic, shock-free  flows;
		\item[(iv)] There are solutions with the properties {\em (i)-(iii)} where in addition the fluid
		is everywhere flowing toward the center of motion at time $t=0$. Similarly, there are other solutions 
		satisfying (i)-(iii) in which the fluid is everywhere flowing away from the center at time $t=0$. 
	\end{itemize}
\end{theorem}
%%%%%%%%%%%%%%%%%%%%%%%%%%%%
The rest of the paper is organized as follows. Section \ref{setup} provides the setup 
and records the similarity ODEs. These involve 
the space dimension $n$ and the adiabatic constant $\gamma$, as well as two the
similarity parameters $\kappa$ and $\lambda$. (Throughout we mostly follow the notation
in \cite{laz}.) We also
impose three constraints (C1)-(C3). The first requires locally finite amounts of the 
conserved quantities (cf.\ (iii) above); (C2) prescribes that blowup should occur at a single 
point in space-time (chosen as the origin); and (C3) concerns the types of the relevant 
critical points in the plane of self-similar variables. We also record some standard 
observations about the similarity ODEs and their critical points, as well as 
the existence of a (well-known) exact integral. Finally, it is recorded that 
(C1)-(C3) implies a particular value of one of the similarity 
parameters ($\kappa=\bar\kappa$), and it is verified that this in turn implies 
that the flow is globally isentropic.

Section \ref{crit_pts} records 
the critical points of the similarity ODEs and analyzes their presence in the case
of isentropic flow. The implications of constraints (C1)-(C3) are 
analyzed in Sections \ref{lam_kap_constrs} and \ref{presence_type}. 
Section \ref{lam_kap_constrs} argues for the aforementioned ``isentropic'' 
value of $\kappa=\bar\kappa$ (this argument was also presented in \cites{jj,jt4}), and
Proposition \ref{constraints_C1_C_2_and _2nd_C3} 
summarizes the relevant implications of this choice.

Section \ref{presence_type} deals with the arguments for
how the second similarity parameter $\lambda$ must be further restricted in terms 
of $n$ and $\gamma$ in order to satisfy constraint (C3). 
This takes the form of a series of increasingly restrictive upper bounds on $\lambda$.
Although we do not evaluate analytically all of these constraints,
we show that they are all met whenever $\lambda>1$ is sufficiently close to $1$
(for $n$ and $\gamma$ fixed).
This will suffice for our needs.
Section \ref{further_props} provides further properties of the critical points;
see Proposition \ref{locn_of_crit_points} for a summary. Up to this point, most 
of the analysis is standard.

The main result, i.e.\ existence of radial self-similar Euler flows 
that suffer amplitude blowup at a single space-time point without
generating a shock wave, is proved in Section \ref{global_flow}.
The is accomplished by exploiting one particular trajectory of
the similarity ODEs, viz.\ the one that passes {\em vertically} through the
origin in the $(V,C)$-plane. An argument based on barriers demonstrates 
that for $n=3$ and any $\gamma>1$, 
this trajectory joins the two nodal points $P_8$ and $P_9$ via the origin $P_1$, 
whenever $\lambda>1$
is sufficiently close to $1$. In turn, $P_8$ and $P_9$ are connected by trajectories
to two critical points at infinity ($P_{+\infty}$ and $P_{-\infty}$, respectively).
Joining the various trajectories, we obtain a global solution of the similarity ODEs, 
and this  yields the sought-for globally defined, shock-free 3-d Euler flow as 
described in Theorem \ref{thm}.

The analysis will show that the particular trajectory under consideration
approaches the nodes $P_8$ and $P_9$ along their primary direction. 
As a consequence there are nearby trajectories
that, while still connecting $P_8$ to $P_9$, pass {\em non-vertically} 
through the origin. These ``perturbed'' trajectories will have either positive
or negative slopes at the origin in the $(V,C)$-plane. As a consequence, there are 
3-d Euler flows which suffer amplitude blowup at $(t,r)=(0,0)$, 
while remaining continuous everywhere else, 
and with the fluid at time $t=0$ everywhere moving either toward or away from
the center of motion (cf., part (iv) of Theorem \ref{thm}).

The fact that blowup (including of density) can occur also when all particles 
are moving away from the origin might seem counterintuitive. 
It is perhaps more surprising that a shock is not necessarily 
generated even when the pressure blows up at the center of motion and all fluid 
particles move toward the origin at time of collapse.
These issues are briefly discussed in Section \ref{remarks_1}.

Finally, in Section \ref{remarks_2}, we consider the issue of non-uniqueness (our original 
motivation). Namely, with blowup solutions of the type discussed above, one is really 
pushing the Euler system to (or beyond) its limits as a physical model. 
It is reasonable to ask if such solutions could be used to give examples
of non-uniqueness. With solutions as described by Theorem \ref{thm},
where in particular the pressure field becomes unbounded at a point, 
it would be reasonable to expect that an expanding shock should form. So, could it be
that, upon freezing time at $t=0$, the data of the constructed blowup solutions in Theorem \ref{thm} 
also admit a propagation for $t>0$ which is discontinuous, and thus distinct from the
solution described in Theorem \ref{thm}? Section \ref{remarks_2} outlines 
an analysis, based on the Rankine-Hugoniot relations, which indicates 
(but does not conclusively prove) that the answer is negative.

%%%%%%%%%%%%%%%%%%%%%%%%%%%%%%%%%%%%%
\begin{remark}
	Concerning vacuum formation (vanishing of the density field) we have that:
	\begin{itemize}
		\item The solutions described in Theorem \ref{thm} do not exhibit
		vacuum formation; in particular, the density remains strictly 
		positive at the center of motion for all times.
		On the other hand, vacuum is approached asymptotically as $r\to\infty$ at all times,
		and also along the center of motion as $t\to\pm\infty$; see Remarks 
		\ref{far_field_asymp} and \ref{no_vacuum}.
		\item This is in contrast to Guderley shock solutions, i.e., radial, self-similar
		Euler flows in which an incoming shock propagates into a quiescent fluid near $r=0$ 
		for $t<0$, \cites{gud,laz}. For such solutions, a one-point
		vacuum is present at the center of motion $r=0$ for all times 
		following collapse; see \cite{jt1} (observation (O2), Section IV.C)
		and also the numerical plots in  \cite{laz} (Figures 8.27–8.28). 
		A one-point vacuum also occurs at and after collapse
		in the solutions considered in \cite{jj}.
	\end{itemize}
\end{remark}
%%%%%%%%%%%%%%%%%%%%%%%%%%%%%%%%%%%%%

%%%%%%%%%%%%%%%%%%%%%%%%%%%%%%%%%%%
%%%%%%%%%%%%%%%%%%%%%%%%%%%%%%%%%%%
\section{Similarity variables and similarity ODEs}\label{setup}
%%%%%%%%%%%%%%%%%%%%%%%%%%%%%%%%%%%
%%%%%%%%%%%%%%%%%%%%%%%%%%%%%%%%%%%
The overall goal is to demonstrate existence of
solutions to \eq{m_eul}-\eq{mom_eul}-\eq{ener_eul}, 
with certain properties, for all times $t\in(-\infty,\infty)$ and all $r>0$. This will be 
accomplished by building suitable solutions to the system of similarity ODEs
\eq{V_sim2}-\eq{C_sim2} obtained by substituting the ansatz \eq{sim_vars} into 
\eq{m_eul}-\eq{mom_eul}-\eq{ener_eul}. 
As we seek globally defined Euler flows, we need to 
build solutions of the similarity ODEs \eq{V_sim2}-\eq{C_sim2} that are defined for 
all $x\in(-\infty,\infty)$.

We shall impose several constraints on the solutions of interest to us,
and this will severely restrict the possible values that $\kappa$ and 
$\lambda$ can take.
These constraints are:
\begin{enumerate}
	\item[(C1)] The total amounts of mass, momentum, and energy
	in any fixed ball about the origin in physical space should remain finite 
	at all times.
	\item[(C2)] The solutions are to suffer amplitude blowup 
	at the center of motion at time of 
	collapse $t=0$, and at no other place or time.
	In particular, $\rho(t,r)$, $u(t,r)$, $c(t,r)$ (and 
	hence also $\theta(t,r)$ and $p(t,r)$, cf.\ \eq{pressure1}-\eq{sound_speed})
	should remain bounded as $r\downarrow0$ whenever $t\neq0$.
	\item[(C3)] Two of the critical points of the similarity ODEs 
	\eq{V_sim2}-\eq{C_sim2} (points $P_8$ and $P_9$ in what follows) 
	should be present as nodes, while its two critical points at infinity 
	($P_{\pm\infty}$) should be saddles that are approached as $|x|\to\infty.$
\end{enumerate}
The precise implications of (C1)-(C3) will be detailed in later sections.
For now we note that (C1) is a minimal requirement for physicality, 
while (C2) is imposed to focus attention on highly singular solutions 
that suffer amplitude blowup of primary flow variables. 
All solutions we consider will be such that both $\frac{V(x)}{x}$ and 
$\frac{C(x)}{x}$ tend to finite limits as $x\to 0$. By sending $t\to0$ with 
$r>0$ fixed in \eq{sim_vars}, we will therefore  have 
\beq\label{uc}
	u(0,r), c(0,r)\propto r^{1-\lambda}.
\eeq
In particular, as (C2) requires blowup of $u(0,r), c(0,r)$ as $r\downarrow0$, 
we shall only consider $\lambda>1$. On the other hand, we shall exploit the 
fact that blowup occurs whenever $\lambda>1$: our main result on 
continuous blowup will be established for $\lambda$ sufficiently close to $1$
(depending on $\gamma$).
%%%%%%%%%%%%%%%%%%%%%%%%%%%%%%%%%%%%%%
%\begin{remark} 
%	The range $0<\lambda<1$ was recently considered in \cite{jj}, and 
%	\eq{uc} shows that the solutions have a milder singularity in this case:
%	the primary flow variables remain bounded 
%	near $r=0$, but suffer {\em gradient blowup} as $r\downarrow0$ at time 
%	of collapse. [REFORMULATE]
%\end{remark}
%%%%%%%%%%%%%%%%%%%%%%%%%%%%%%%%%%%%%%

Constraint (C3) is imposed in order to argue that there
are unique trajectories of the reduced similarity ODE (\eq{CV_ode} below)
connecting the critical points $P_8$ and $P_9$ to its critical points $P_{\pm\infty}$ 
at infinity. It will also guarantee, for $\lambda\gtrsim1$, that there are infinitely many trajectories of the 
similarity ODEs connecting $P_8$ and $P_9$ via the origin in the $(V,C)$-plane.

We note that for Guderley solutions, the similarity parameter $\kappa$ must necessarily vanish 
(see p.\ 318 in \cite{laz}). In contrast, for the continuous flows considered in this work, 
(C1)-(C3) will imply that $\kappa$ must take a 
non-zero value which is given in terms of $\lambda$ and $\gamma$ (cf.\ \ref{isentr_kappa}). 
This turns out to imply that the flows we construct are in fact isentropic:
the specific entropy is globally constant in space and time, cf.\ Section \ref{lam_kap_constrs}.

%%%%%%%%%%%%%%%%%%%%%%%%%%%%%%%%%%%%%
\begin{remark}
	Theorem \ref{thm} is formulated for 3-dimensional flows with spherical 
	symmetry. However, most of the analysis applies also to the case of cylindrical
	symmetry $n=2$. As will be clear from what follows, 
	it is only one part of the final barrier argument
	in Section \ref{prop_A} that requires $n=3$. It would be of interest to know if
	the extra dimension is really required for the blowup behavior described 
	in Theorem \ref{thm}.
\end{remark}
%%%%%%%%%%%%%%%%%%%%%%%%%%%%%%%%%%%%%

%%%%%%%%%%%%%%%%%%%%%%%%%%%%%%%%%%%%%
\subsection{Similarity ODEs and reduced similarity ODE}\label{ss_odes}
%%%%%%%%%%%%%%%%%%%%%%%%%%%%%%%%%%%%%
The self-similar ansatz \eq{sim_vars} 
reduces \eq{m_eul}-\eq{mom_eul}-\eq{ener_eul}
to a system of three {\em similarity ODEs} for $R(x)$, $V(x)$, $C(x)$.
A calculation shows that these are given by ($'\equiv\frac{d}{dx}$)
\begin{align}
	(1+V)R'+RV'&=\textstyle\frac{\kappa+n}{\lambda x}RV \label{R_sim1}\\
	C^2R'+\gamma R(1+V)V'+2RCC'&=
	\textstyle\frac{1}{\lambda x}[\gamma(\lambda+V)V+(\kappa+2)C^2]R \label{V_sim1}\\
	\textstyle\frac{\gamma-1}{2}CV'+(1+V)C'&=\textstyle\frac{1}{\lambda x}
	[\lambda+(1+\textstyle\frac{n(\gamma-1)}{2})V]C, \label{C_sim1}
\end{align}
where we observe that \eq{C_sim1} does not involve $R$.
An essential observation, due Guderley \cite{gud}, is that $R$ can be 
eliminated from \eq{V_sim1} by using \eq{R_sim1}. We thus obtain two  
ODEs for only $V$ and $C$. Solving these for $V'$ and $C'$ yields
\begin{align}
	V'&=-\textstyle\frac{1}{\lambda x}\frac{G(V,C)}{D(V,C)}\label{V_sim2}\\
	C'&=-\textstyle\frac{1}{\lambda x}\frac{F(V,C)}{D(V,C)},\label{C_sim2}
\end{align}
with 
\begin{align}
	D(V,C)&=(1+V)^2-C^2\label{D}\\
	G(V,C)&=nC^2(V-V_*)-V(1+V)(\lambda+V)\label{G}\\
	F(V,C)&=C\big[C^2\big(1+\textstyle\frac{\alpha}{1+V}\big)
	-k_1(1+V)^2+k_2(1+V)-k_3\big],\label{F}
\end{align}
where
\beq\label{V_*}
	V_*=\textstyle\frac{\kappa-2(\lambda-1)}{n\gamma},
\eeq
\beq\label{alpha}
	\alpha=\textstyle\frac{1}{\gamma}[(\lambda-1)+\frac{\kappa}{2}(\gamma-1)],
\eeq
and
\beq\label{ks}
	k_1=1+{\textstyle\frac{(n-1)(\gamma-1)}{2}},\qquad 
	k_2={\textstyle\frac{(n-1)(\gamma-1)+(\gamma-3)(\lambda-1)}{2}},\qquad
	k_3=\textstyle\frac{(\gamma-1)(\lambda-1)}{2}.
\eeq
We note that 
\beq\label{sum_k}
	k_1-k_2+k_3=\lambda.
\eeq
Finally, \eq{V_sim2}-\eq{C_sim2} give the single {\em reduced similarity ODE}
\beq\label{CV_ode}
	\frac{dC}{dV}=\frac{F(V,C)}{G(V,C)}
\eeq
which relates $V$ and $C$ along self-similar Euler flows.

%%%%%%%%%%%%%%%%%%%%%%%%%%%%%%%%%%%%%
\subsection{Strategy and preliminary observations}\label{strat_prelim}
%%%%%%%%%%%%%%%%%%%%%%%%%%%%%%%%%%%%%
We start by noting that, although the phase plane analysis of \eq{CV_ode}
is central to the analysis, what is really required for building
Euler flows of the form \eq{sim_vars},  are suitable solutions 
of \eq{V_sim2}-\eq{C_sim2}. And the latter system is not equivalent
to the single ODE \eq{CV_ode}: In contrast to \eq{CV_ode}, the system 
\eq{V_sim2}-\eq{C_sim2} degenerates along the {\em critical (sonic) lines}
\beq\label{crit_lines}
	L_\pm:=\{C=\pm(1+V)\},
\eeq
across which the numerator $D$ in \eq{V_sim2}-\eq{C_sim2} changes sign.
A direct calculation verifies that
\beq\label{non_obvious_reln}
	F(V,\pm(1+V))\equiv \mp\textstyle\frac{(\gamma-1)}{2}G(V,\pm(1+V)),
\eeq
i.e., $F$ and $G$ are proportional along $L_\pm$.
Therefore, if a trajectory $\Gamma$ of the reduced similarity ODE \eq{CV_ode} 
crosses one of the critical lines 
$L_\pm$ at a point $P$ where one (and hence both) of $F$ and $G$ 
is non-zero, then the flow of \eq{V_sim2}-\eq{C_sim2} along $\Gamma$ is directed in 
opposite directions on either side of the critical line at $P$. The upshot is that 
such a trajectory 
$\Gamma$ cannot be used to generate a physically meaningful Euler flow. 

This implies a drastic reduction of the set of relevant trajectories: Any continuous crossing 
of one of the critical lines $L_\pm$ by a trajectory of  \eq{CV_ode} must occur at 
a ``triple point'' where $F=G=D=0$. As we shall see, the existence of triple points off the $V$-axis
amounts to the statement that the aforementioned critical points $P_8$ and $P_9$ 
are present.

%%%%%%%%%%%%%%%%%%%%%%%%%%%%%%%%%%%%%%
\subsubsection{Strategy and an exact integral}\label{strat_exact_int}
%%%%%%%%%%%%%%%%%%%%%%%%%%%%%%%%%%%%%%
We can now be somewhat more precise about the general strategy for building 
continuous radial self-similar Euler flows. Assuming we have identified a suitable solution
$C(V)$ of the reduced similarity ODE \eq{CV_ode} (in particular, crossing $L_\pm$ 
at triple points), we solve \eq{V_sim2}
with $C=C(V)$ to obtain $V=\hat V(x)$. Assuming further that this yields a global 
solution (i.e, $\hat V(x)$ is defined for all $x\in(-\infty,\infty)$), we obtain 
a globally defined $\hat C(x):=C(\hat V(x))$. In turn, $\hat V(x)$ and $\hat C(x)$
define, via \eq{sim_vars}, the flow variables $u(t,r)$ and $c(t,r)$ in an Euler flow.

To determine the density field $\rho(t,r)$, we may solve 
\eq{R_sim1} (with $V=\hat V(x)$) for $R=\hat R(x)$.
However, the similarity ODEs \eq{R_sim1}-\eq{C_sim1} admit an 
exact ``adiabatic'' integral (see Eqn.\ (2.7) in \cite{laz}, or pp.\ 319-320 in \cite{rj}).
This may be deduced as follows. For an ideal polytropic gas, the specific entropy
$S$ is a function of $\rho^{1-\gamma}\theta\propto \rho^{1-\gamma}c^2$. 
For a self-similar flow \eq{sim_vars} it follows that
\beq\label{S_sigma}
	S=\text{function of }r^{-2\gamma\alpha}\sigma(x)\qquad \text{where}\quad
	\sigma(x):=R(x)^{1-\gamma}(\textstyle\frac{C(x)}{x})^2,
\eeq
and $\alpha$ is given by \eq{alpha}.
Equation \eq{spec_entr} therefore becomes
\[(1+V)\sigma'+2\gamma\alpha\textstyle\frac{V}{\lambda x}\sigma=0.\]
Multiplying through by $R$, and substituting for $\frac{RV}{\lambda x}$ from 
the right-hand side of \eq{R_sim1}, results in
\[R(1+V)\sigma'+\textstyle\frac{2\gamma\alpha}{\kappa+n}[R(1+V)]'\sigma=0.\]
It follows that
\beq\label{adiab_int}
	[R|1+V|]^qR(x)^{1-\gamma}(\textstyle\frac{C(x)}{x})^2\equiv const. >0,
\eeq  
where 
\beq\label{q}
	q=\textstyle\frac{2\gamma\alpha}{\kappa+n}\equiv
	\frac{1}{\kappa+n}[\kappa(\gamma-1)+2(\lambda-1)].
\eeq 
%Following \cite{rj}, we refer to the exact integral \eq{adiab_int} as the {\em adiabatic} integral.
%We note that the constant value of
%the adiabatic integral will in general change across discontinuities in the flow.
For later use, we record the following special case:
When $\kappa$ takes the value
\beq\label{isentr_kappa}
	\bar\kappa=\bar\kappa(\gamma,\lambda):=-\textstyle\frac{2(\lambda-1)}{\gamma-1},
\eeq
the constants $\alpha$ and $q$ vanish, and \eq{adiab_int} reduces to
\beq\label{CR}
	(\textstyle\frac{C(x)}{x})^2R(x)^{1-\gamma}\equiv const. >0.
\eeq
According to \eq{S_sigma}, this means that the specific entropy $S$ in this case
takes a constant value throughout any continuous 
part of the flow. Thus, when we construct a globally continuous self-similar Euler 
flow \eq{sim_vars} with $\kappa=\bar\kappa$, we necessarily obtain an isentropic flow
(i.e., $S$ takes on a constant value globally in space-time). 
We therefore refer to $\bar\kappa$ as the {\em isentropic} $\kappa$-value. We
shall see in Section \ref{lam_kap_constrs} that this value is in fact a 
consequence of the constraints (C1)-(C3) above.

%%%%%%%%%%%%%%%%%%%%%%%%%%%%%%%%%%%%%%
\subsubsection{Preliminaries about critical points}\label{prelim_crit_points}
%%%%%%%%%%%%%%%%%%%%%%%%%%%%%%%%%%%%%%
To build our self-similar flows we need to analyze the critical points of \eq{CV_ode} 
in some detail. These are the common zeros of the functions $F$ and $G$, 
and their number and locations depend on $n$, $\gamma$, $\lambda$, 
$\kappa$. It turns out that some of these are necessarily located
on the critical lines (cf.\ \eq{crit_lines}); they thus provide triple 
points where trajectories of \eq{CV_ode} can cross $L_\pm$.

The zero-levels of $F$ and $G$ are denoted
\[\mathcal F=\{(V,C)\,|\,F(V,C)=0\}\qquad\mathcal G=\{(V,C)\,|\,G(V,C)=0\},\]
and we note that $V=V_*$ (cf.\ \eq{V_*}) is a vertical asymptote for $\mathcal G$.
Finally, we record the symmetries
\beq\label{symms}
	G(V,-C)=G(V,C)\qquad\text{and}\qquad F(V,-C)=-F(V,C).
\eeq
In particular, trajectories of \eq{CV_ode} in the $(V,C)$-plane are located 
symmetrically about the $V$-axis, and critical points off the $V$-axis 
come in symmetric pairs.

%%%%%%%%%%%%%%%%%%%%%%%%%%%%%%%%%%%%
%%%%%%%%%%%%%%%%%%%%%%%%%%%%%%%%%%%%
\section{Critical points of the reduced similarity ODE}\label{crit_pts}
%%%%%%%%%%%%%%%%%%%%%%%%%%%%%%%%%%%%
%%%%%%%%%%%%%%%%%%%%%%%%%%%%%%%%%%%%
A complete recording of the critical points of \eq{CV_ode},
for all values of $(n,\gamma,\lambda,\kappa)$,  was recently 
given in \cite{jj}, and we shall only review the relevant findings.
We assume $n=2$ ($m=1$) or $n=3$ ($m=2$) and $\gamma>1$. Until 
further notice, $\lambda$ and $\kappa$ are considered as free parameters.

There are up to nine critical points of \eq{CV_ode}, and these come in 
two groups (notation as in \cites{jj,laz}):
$P_1$-$P_3$ located along the $V$-axis, and three pairs $P_4$-$P_5$, 
$P_6$-$P_7$, and $P_8$-$P_9$, with each pair of points located off and,
according to \eq{symms}, symmetrically about the $V$-axis.
There are two further critical points at infinity, viz.
\beq\label{p+-}
	P_{\pm\infty}:=(V_*,\pm\infty).
\eeq
The points $P_{\pm\infty}$,  analyzed in Section \ref{P_infs} below, 
are central to the construction of continuous Euler 
flows: These are the states that the relevant solutions $(V(x),C(x))$ of 
\eq{V_sim2}-\eq{C_sim2} tend to as $x\to\mp\infty$, respectively.

%%%%%%%%%%%%%%%%%%%%%%%%%%
\subsection{Critical points on the $V$-axis}\label{P1_P3}
%%%%%%%%%%%%%%%%%%%%%%%%%%
From \eq{G}-\eq{F} we have $F(V,0)\equiv0$ and $G(V,0)=V(1+V)(\lambda+V)$.
Therefore, there are in general three critical points located on the $V$-axis, viz.\
\[P_1=(0,0)\qquad P_2=(-1,0)\qquad P_3=(-\lambda,0),\] 
of which only $P_1$ is relevant in the present work. 
The linearization of 
\eq{CV_ode} at $P_1$ is $\frac{dC}{dV}=\frac C V$. Thus, $P_1$ is a 
star point: Whenever $\mathcal L$ is a straight half-line from 
$P_1$, there is a unique trajectory $(V,C(V))$ of \eq{CV_ode} that 
approaches $P_1$ tangent to $\mathcal L$ (cf.\ Theorem 3.6, pp.\ 218-219 in \cite{hart}).

We need to analyze how the corresponding solutions $(V(x),C(x))$ of the system
\eq{V_sim2}-\eq{C_sim2} behave as $P_1$ is approached. 
Assume $\mathcal L$ has slope $\ell$, so that $C(V)\approx \ell V$ for $V\approx 0$
along the solution in question. Substitution into \eq{V_sim2}-\eq{C_sim2}, and 
use of \eq{sum_k}, show that
\[\frac{dV}{dx}\approx \frac V x \qquad\text{and}\qquad \frac{dC}{dx}\approx \frac C x
\qquad\text{as $P_1$ is approached.}\]
It follows that any solution to \eq{V_sim2}-\eq{C_sim2} which 
approaches $P_1$ must do so as $x\to 0$. Furthermore, the limits
\beq\label{well_bhvd}
	\nu=\lim_{x\to0}{\textstyle\frac{V(x)}{x}}\qquad\text{and}\qquad
	\omega=\lim_{x\to0}\textstyle\frac{C(x)}{x}\qquad\text{exist as finite numbers, 
	and $\ell=\textstyle\frac{\omega}{\nu}$.}
\eeq
%\begin{remark}
	With these notations we can specify precisely the velocity and sound speed
	at time of collapse in the corresponding Euler flow (cf.\ \eq{uc}). Sending $t\to0$
	with $r>0$ fixed, \eq{sim_vars} and \eq{well_bhvd} give
	\beq\label{uc_at_collapse}
		u(0,r)= -\textstyle\frac{\nu}{\lambda}r^{1-\lambda}
		\qquad\text{and}\qquad
		c(0,r)=-\textstyle\frac{\omega}{\lambda} r^{1-\lambda}.
	\eeq
%\end{remark}
We note that $\omega=0$ gives $\ell=0$, which corresponds to the two flat
trajectories $C(V)\equiv 0$, for $V\gtrless0$, of \eq{CV_ode}. 
Also, the limiting case $\nu=0$ gives
$\ell=\pm\infty$, which corresponds to the two trajectories of \eq{CV_ode} that 
reach $P_1$ vertically. 
The latter trajectories will be of particular interest to us.

%%%%%%%%%%%%%%%%%%%%%%%%%%%%%%%%%%%%
\begin{remark}\label{far_field_asymp}
	The behavior of $C(x)$ as $x\to0$ yields information about the 
	asymptotic behavior of the sound speed and density in the far field as $r\to\infty$. 
	In particular, with $\kappa=\bar\kappa$ (isentropic flow) we obtain from 
	\eq{sim_vars}${}_4$ and \eq{well_bhvd}${}_2$ that
	\beq\label{far_field}
		\rho(t,r)\propto c(t,r)^\frac{2}{\gamma-1}\sim (\textstyle\frac{|\omega|}{\lambda})^\frac{2}{\gamma-1}r^{\bar\kappa}
		\qquad\text{as $r\to\infty$.}
	\eeq
	Since $\bar\kappa<0$ for $\lambda>1$, this shows that the density decays to zero 
	in the far field for the solutions described in Theorem \ref{thm}.
\end{remark}
%%%%%%%%%%%%%%%%%%%%%%%%%%%%%%%%%%%%

We next analyze how 
solutions to \eq{V_sim2}-\eq{C_sim2} pass through $P_1$.
Recall that the sound speed $c$ is non-negative by definition
(cf.\ \eq{sound_speed}), and that we only consider positive $\lambda$-values
(in fact, $\lambda>1$). It therefore follows from \eq{sim_vars}${}_4$
that a solution $(V(x),C(x))$ of \eq{V_sim2}-\eq{C_sim2} which passes through the
origin must do so by moving from the upper half-plane $\{C>0\}$ to the lower half-plane
$\{C<0\}$ as $x$ increases from negative to positive values.
In particular, this shows that $\omega\leq0$. In fact, since $\omega=0$
corresponds to the flat trajectories $C(V)\equiv 0$ ($V\gtrless0$), we have 
$\omega<0$ for any trajectory considered in the following.

On the other hand, the sign of $\nu$ depends on whether the origin $P_1$
is approached from within the 1st or 2nd quadrant in the $(V,C)$-plane. For later reference
we note that, according to \eq{uc_at_collapse}${}_1$, the sign of $\nu$ determines 
the direction of the flow at time of collapse $t=0$:
\begin{itemize}
	\item Case 1: If $\ell=\frac{\omega}{\nu}>0$, i.e., the solution 
	approaches $P_1$ from within the 1st quadrant, then $\nu<0$ and $u(0,r)>0$ 
	for all $r>0$. Thus $\ell>0$ corresponds to an everywhere diverging flow at 
	time of collapse.
	\item Case 2:  If $\ell=\frac{\omega}{\nu}<0$, i.e., the solution 
	approaches $P_1$ from within the 2nd quadrant, then $\nu>0$ and $u(0,r)<0$ 
	for all $r>0$. Thus $\ell<0$ corresponds to an everywhere converging flow at 
	time of collapse.
\end{itemize}
These observations will be relevant for  part (iv) of Theorem \ref{thm}; 
see Section \ref{part_iv}.
%%%%%%%%%%%%%%%%%%%%%%%%%%
\begin{remark}
	As we shall see below, the solutions $(V(x),C(x))$ we construct
	connect the critical point $P_8$ to the origin $P_1$.
	Since $P_8$ will be located in the 2nd quadrant (see Proposition
	\ref{locn_of_crit_points}), we obtain that if Case 1 occurs, then the solution 
	trajectory must cross the $C$-axis on its way from $P_8$ to $P_1$.
	Thus, $V(x)$ must vanish for some $x=\bar x<0$, which means that the
	corresponding flow exhibits stagnation, i.e., vanishing velocity, along the path
	$\bar r(t)=(\frac{t}{\bar x})^\frac{1}{\lambda}$ prior to collapse ($t<0$).
	
	Similarly, in Case 2, the flows we consider will necessarily exhibit stagnation 
	after 	collapse.
\end{remark}
%%%%%%%%%%%%%%%%%%%%%%%%%%

The issue of blowup with a diverging or converging flow at time of collapse
is addressed further in Section \ref{remarks_1}.

%%%%%%%%%%%%%%%%%%%%%%%%%%
%%%%%%%%%%%%%%%%%%%%%%%%%%
\subsection{Critical points off the $V$-axis}\label{P4_P9}
%%%%%%%%%%%%%%%%%%%%%%%%%%
%%%%%%%%%%%%%%%%%%%%%%%%%%
The critical points $P_i=(V_i,C_i)$, $i=4,\dots,9$, are identified by 
first solving for $C^2$ in terms of $V$ from $G(V,C)=0$, i.e.,
\beq\label{g}
	C^2=g(V):=\textstyle\frac{V(1+V)(\lambda+V)}{n(V-V_*)},
\eeq
and then using the result in $F(V,C)=0$. 
This gives a cubic equation for $V$ which has one real root
$V_4=V_5$ for any choice of $n,\gamma,\lambda,\kappa$. The remaining 
$V$-quadratic has real roots $V_6=V_7$ and $V_8=V_9$ only in certain 
cases, and these give the critical points $P_6$-$P_9$ (when present).

Referring to \cite{jj} for details, we only record the results.
For labeling we follow Lazarus \cite{laz} and let $P_4$, $P_6$, $P_8$ be 
the critical points located in the upper half-plane $\{C>0\}$, while the 
points $P_5$, $P_7$, $P_9$, respectively, are located symmetrically about 
the $V$-axis in the lower half-plane.

The points $P_4=(V_4,C_4)$ and $P_5=(V_4,-C_4)$ are given by
\beq\label{V_4_C_4}
	V_4=-\textstyle\frac{2\lambda}{2+n(\gamma-1)},\qquad
	C_4=\sqrt{g(V_4)}.
\eeq
Note that $g(V_4)$ may be negative, in which case $P_4$ and $P_5$ 
are not present. However, we shall see that they are necessarily present whenever 
the constraints (C1)-(C3) are met (see Lemma \ref{pres_locn}).

It is convenient to set 
\[V_6=V_7\equiv V_-\qquad\text{and} \qquad V_8=V_9\equiv V_+.\] 
A calculation shows that these are given by
\begin{align}
	&V_\pm=\textstyle\frac {1}{2 m \gamma} 
	\Big[(\gamma-2)\mu +\kappa-m\gamma\pm\sqrt{(\gamma-2)^2\mu^2
	-2[\gamma m(\gamma+2)-\kappa(\gamma-2)]\mu+(\gamma m+\kappa)^2}\Big],\label{V_pm}
\end{align}
where $m=n-1$, and $\mu=\lambda-1$.
%%%%%%%%%%%%%%%%%%%%%%%%%%%%%%%%%%%%%%%
\begin{notation}\label{mmueps}
	To shorten some of the expressions in what follows, we will freely use either
	$n$ or $m$, $\lambda$ or $\mu$, and $\gamma$ or $\eps$, always assuming
	\[m=n-1,\qquad \mu=\lambda-1,\qquad \eps=\gamma-1.\]
	Note that our standing assumptions imply $m\geq1$, $\mu>0$, and $\eps>0$.
\end{notation}
%%%%%%%%%%%%%%%%%%%%%%%%%%%%%%%%%%%%%%%
It follows from \eq{g} and \eq{V_pm} that the critical points $P_6$-$P_9$ are present if and 
only if, first, the radicand in \eq{V_pm} is positive, and second, $g(V_\pm)>0$. 
If so, we get  the critical points 
\beq\label{P_6-P_9}
	P_6=(V_-,C_-), \qquad P_7=(V_-,-C_-), \qquad P_8=(V_+,C_+), \qquad 
	\text{and}\qquad P_9=(V_+,-C_+),
\eeq
where 
\[C_\pm=\sqrt{g(V_\pm)}.\] 
Note that $P_6$ ($P_7$) coalesces with
$P_8$ ($P_9$, respectively) whenever $V_+=V_-$. This occurs for certain values
of the parameters $n$, $\gamma$, $\lambda$; see (A)-(B) below.

A direct calculation reveals that 
\beq\label{on_L_+-}
	C_\pm^2=(1+V_\pm)^2,
\eeq
so that $P_6$-$P_9$, when present, are located along the critical lines $L_\pm$.

%%%%%%%%%%%%%%%%%%%%%%%%%%
\subsection{$P_6$-$P_9$ in isentropic case}
%%%%%%%%%%%%%%%%%%%%%%%%%%
We proceed to analyze the presence of $P_6$-$P_9$ in the special case 
when $\kappa$ takes the isentropic value $\bar\kappa$ in \eq{isentr_kappa} (as will be seen 
to be dictated by the constraints (C1)-(C3)):
\[\bar\kappa=-\textstyle\frac{2(\lambda-1)}{\gamma-1}\equiv-\textstyle\frac{2\mu}{\eps}.\]
In this case \eq{V_pm} gives
\beq\label{isntr_V_pm}
	V_\pm=\textstyle\frac{1}{2}(\mathfrak a_n\pm\sqrt{\mathfrak Q_n}),
\eeq
where 
\beq\label{aQ}
	\mathfrak a_n=\mathfrak a_n(\gamma,\lambda)=\textstyle\frac{(\eps-2)}{m\eps}\mu-1\qquad\text{and}\qquad
	\mathfrak Q_n=\mathfrak Q_n(\gamma,\lambda)=\left(\textstyle\frac{\eps-2}{m\eps}\right)^2\!\!\mu^2
	-2\textstyle\frac{(\eps+2)}{m\eps}\mu+1.
\eeq
A necessary condition for $P_6$-$P_9$ to be present is that $\mathfrak Q_n\geq 0$ in \eq{isntr_V_pm},
and there are two cases:
\begin{enumerate}
	\item[(A)] If $\gamma\neq3$, then $\mathfrak Q_n$ is quadratic in $\mu$, and a calculation
	shows that $\mathfrak Q_n\geq0$ if and only if
	$\lambda\leq \hat\lambda_n(\gamma)$ or $\lambda\geq\check\lambda_n(\gamma)$, where 
	\beq\label{lambda_max}
		\hat\lambda_n(\gamma):=1+\textstyle\frac{(n-1)(\gamma-1)}{(\gamma+1)+\sqrt{8(\gamma-1)}}
	\eeq
	and 
	\beq\label{lambda_min}
		\check\lambda_n(\gamma)=1+\textstyle\frac{(n-1)(\gamma-1)}{(\gamma+1)-\sqrt{8(\gamma-1)}}.
	\eeq
	In this case, $V_+=V_-$ if and only if either $\lambda=\hat\lambda_n(\gamma)$ or $\lambda=\check\lambda_n(\gamma)$.
	\item[(B)] If $\gamma=3$, then $\mathfrak Q_n$ is linear in $\mu$, and a calculation shows 
	that $\mathfrak Q_n\geq 0$ if and only if $\lambda\leq\hat\lambda_n(3)\equiv1+ \frac{n-1}{4}$. 
	Also, $V_+=V_-$ if and only if $\lambda=\hat\lambda_n(3)$.
\end{enumerate}
We shall later provide sufficient conditions for the presence of $P_6$-$P_9$; see Section \ref{P_8_9_locns}.

%%%%%%%%%%%%%%%%%%%%%%%%%%
\subsection{Critical points at infinity}\label{P_infs}
%%%%%%%%%%%%%%%%%%%%%%%%%%
The critical points at infinity for \eq{CV_ode}, viz.\ $P_{\pm\infty}=(V_*,\pm\infty)$,
will serve as end states at $x=\mp\infty$, respectively, of the sought-for solutions 
$(V(x),C(x))$ to \eq{V_sim2}-\eq{C_sim2}. According to constraint (C3), we shall 
want there to be a unique $P_9P_{-\infty}$-trajectory of \eq{CV_ode}.
(From \eq{symms} it follows that there is then a unique $P_{+\infty}P_8$-trajectory 
as well, viz.\ the reflection of the $P_9P_{-\infty}$-trajectory about the $V$-axis.) At the same time,
we want there to be infinitely many $P_8P_9$-trajectories that pass through $P_1=(0,0)$.
We therefore require that $P_8$ and $P_9$ are nodes and that $P_{\pm\infty}$ are saddles
for \eq{CV_ode} (cf. constraint (C3)).

The critical points at infinity are most easily analyzed by switching to the 
variables $w:=V-V_*$ and $z:=C^{-2}$ (thus treating  $P_{\pm\infty}$ simultaneously).
Linearizing the resulting equation for $\frac{dz}{dw}$ about $(w,z)=(0,0)$ yields
\beq\label{wz_ode}
	\frac{dz}{dw}=\frac{Az}{Bz-nw},
\eeq
with 
\beq\label{ab}
	A=2(1+\textstyle\frac{\alpha}{1+V_*}),\qquad B=V_*(1+V_*)(\lambda+V_*).
\eeq
%where $\alpha$ and $V_*$ are given in \eq{alpha} and \eq{V_*}, respectively.
It follows that $P_{\pm\infty}$ are saddle points if and only if $A>0$. This is clearly  
satisfied when $\kappa=\bar\kappa$ (since then $\alpha=0$ and $A=2$). 
For the remainder of this section $A>0$ is assumed.

As indicated above, we need to argue that there is a unique $P_9P_{-\infty}$-trajectory.
We postpone this argument till Section \ref{9-infty_connecn} when the relative locations 
of $\mathcal F$, $\mathcal G$, and trajectories of \eq{CV_ode} near $P_9$ have been
determined. We therefore take the existence of the $P_9P_{-\infty}$- and $P_{+\infty}P_8$-solutions to 
\eq{V_sim2}-\eq{C_sim2} as given for now, and use \eq{wz_ode} to obtain their behavior 
near $P_{\pm\infty}$. 

The  solutions $(w(x),z(x))$ of \eq{wz_ode} correspond to solutions of 
\eq{V_sim2}-\eq{C_sim2} tending to $P_{\pm\infty}$ 
approach $(0,0)$, and they do so tangent to the stable subspace of \eq{wz_ode} at $(0,0)$.
According to \eq{wz_ode}, this subspace is $\{z=\frac{n+A}{B}w\}$, and it follows that
$C(x)^2(V(x)-V_*)\sim \frac{B}{n+A}$ as $x\to\pm\infty$ for the solutions in question.
Using this information in \eq{V_sim2}, and simplifying the expressions as
$V(x)\to V_*$, shows that to leading order, $V'(x)\sim -\frac{A}{\lambda x}(V(x)-V_*)$.
It follows that the solutions of \eq{V_sim2}-\eq{C_sim2} which approach 
the saddle points $P_{\pm\infty}$, do so to leading order as follows:
\beq\label{VC_asymp}
	|V(x)- V_*|\sim |x|^{-\frac{A}{\lambda}}
	\qquad\text{and}\qquad C(x)\sim\mp|x|^{\frac{A}{2\lambda}}
	\qquad\text{as $x\to\pm\infty$.}
\eeq
%%%%%%%%%%%%%%%%%%%%%%%%%%%%%%%%%%%%
\begin{remark}\label{no_vacuum}
	The asymptotic behavior of $C(x)$ as $|x|\to\infty$ yields information about the 
	values of sound speed and density along the center of motion. In particular,
	with $\kappa=\bar\kappa$ (isentropic flow) we have $A=2$, so that to leading order
	\beq\label{VC_asymp2}
		|C(x)|\sim x^\frac{1}{\lambda}\qquad\text{as $x\to+\infty$.}
	\eeq
	As $\rho\propto c^\frac{2}{\gamma-1}$ in isentropic flow, \eq{VC_asymp2} therefore 
	gives, for a fixed time $t>0$, that
	\[\rho(t,r)\propto c(t,r)^\frac{2}{\gamma-1}
	=\left(\textstyle\frac{r^{1-\lambda}}{\lambda}\frac{|C(x)|}{x}\right)^\frac{2}{\gamma-1}
	\sim k t^\frac{\bar\kappa}{\lambda}\qquad\text{as $r\downarrow 0$,}  \]
	where $k$ is a strictly positive constant.
	This shows that no vacuum occurs along the center of motion at positive 
	times in the solutions described by Theorem \ref{thm} (except asymptotically 
	as $t\to\infty$). The same applies to negative times. 
	
	The only other way vacuum could occur would be by having 
	$\lim_{x\to0}\frac{C(x)}{x}\equiv \omega=0$ (cf.\ \eq{well_bhvd}). 
	However, as argued above in Section \ref{P1_P3}, $\omega<0$ for solutions 
	under consideration.
\end{remark}
%%%%%%%%%%%%%%%%%%%%%%%%%%%%%%%%%%%%%%%%%%%%

%%%%%%%%%%%%%%%%%%%%%%%%%%%%%%%%%%%%
%%%%%%%%%%%%%%%%%%%%%%%%%%%%%%%%%%%%
\section{Restricting $\lambda$ and $\kappa$ in terms of $n$ and $\gamma$}
\label{lam_kap_constrs}
%%%%%%%%%%%%%%%%%%%%%%%%%%%%%%%%%%%%
%%%%%%%%%%%%%%%%%%%%%%%%%%%%%%%%%%%%
In this section and the next we provide analytic formulations and consequences of
the three constraints (C1)-(C3), and see how they delimit the possible 
values of the similarity parameters $\kappa$ and $\lambda$.
Recall the standing assumptions that $n=2$ or $3$, $\gamma>1$, and $\lambda>1$.

Constraint (C1) requires that the integrals 
\beq\label{C1}
	\int_0^{\bar r}\rho(t,r)r^m\, dr,\quad
	\int_0^{\bar r}\rho(t,r)|u(t,r)|r^m\, dr,\quad
	\int_0^{\bar r}\rho(t,r)\left(e(t,r)+\textstyle\frac{1}{2}|u(t,r)|^2\right)r^m\, dr
\eeq
are finite at all times $t$ and for all finite $\bar r>0$.
As we shall see below, the second part of constraint (C2) will imply that the 
solutions we construct are bounded near $r=0$ at all times $t\neq0$. For our present purposes
it therefore suffices to impose (C1) at time of collapse $t=0$. By using \eq{uc_at_collapse} 
together with $\rho(0,r)=R(0)r^\kappa$, we get that \eq{C1} holds at time $t=0$  if 
and only if
\begin{itemize}
	\item[(I)] $\kappa+n>0$
	\item[(II)] $\lambda<1+\kappa+n$
	\item[(III)] $\lambda<1+\textstyle\frac{\kappa+n}{2}$.
\end{itemize}
Clearly, (II) is a consequence of (I) and (III).
For later use we note that (III) and $\gamma>1$ give
\beq\label{1+V*}
	0<\textstyle\frac{n+\kappa-2(\lambda-1)}{n\gamma}
	<\textstyle\frac{n\gamma+\kappa-2(\lambda-1)}{n\gamma}\equiv 1+V_*,
\eeq
where $V_*$ is given in \eq{V_*}.

We proceed to address the second part of constraint (C2). For this we 
assume, in accordance with constraint (C3), that we have selected a
solution $(V(x),C(x))$ of \eq{V_sim2}-\eq{C_sim2} which approaches 
$P_{\pm\infty}$ as $x\to\mp\infty$. In the following argument we fix a
time $t\neq 0$. We then have, by \eq{sim_vars}${}_1$, $|x|\propto r^{-\lambda}$
so that $r\downarrow0$ corresponds to $|x|\to\infty$. 

Now, (C2) requires that the primary flow variables $\rho$, $u$, $c$ remain 
bounded as the center of motion $r=0$ is approached. 
Consider first the speed
\beq\label{u_at_t}
	u(t,r)=-\textstyle\frac{r^{1-\lambda}}{\lambda}\frac{V(x)}{x}=-\frac{r}{\lambda t}V(x)
	\propto V(x)r. 
\eeq
According to \eq{VC_asymp}${}_1$ we have $V(x)\to V_*$ as $|x|\to\infty$,
and it follows that $u(t,r)\sim r$ as $r\downarrow 0$.
I.e., at any time different from the time of collapse $t=0$, the speed of the fluid 
vanishes at a linear rate as the center of motion $r=0$ is approached.
Thus, boundedness (in fact, vanishing) of the fluid velocity as $r\downarrow0$, 
does not imply any further constraints.

Next, consider the sound speed $c(t,r)$ as $r\downarrow 0$. 
From \eq{VC_asymp}${}_2$ and $|x|\propto r^{-\lambda}$ we get
\beq\label{c_at_t}
	c(t,r)=-\textstyle\frac{r}{\lambda t}C(x)
	\sim r^{1-\frac{A}{2}}\qquad\text{as $r\downarrow0$}. 
\eeq 
Thus, boundedness of $c(t,r)$ as $r\downarrow0$ requires $A\leq2$.
According to \eq{ab}${}_1$ and \eq{1+V*}, this amounts to 
\beq\label{alfa}
	\alpha\leq 0.
\eeq
Finally, consider the density field $\rho(t,r)$ as as $r\downarrow 0$. 
Using \eq{sim_vars}${}_2$, the adiabatic integral \eq{adiab_int}, \eq{1+V*}, and \eq{VC_asymp}${}_1$,
we have
\[\rho(t,r)=r^\kappa R(x)\sim r^\kappa(\textstyle\frac{x}{C(x)})^\frac{2}{1-\gamma+q},\]
where $q$ is given in \eq{q}. Using again \eq{VC_asymp}${}_2$ and $|x|\propto r^{-\lambda}$, we 
get
\[\rho(t,r)\sim r^{\kappa+\frac{A-2\lambda}{1-\gamma+q}}\qquad\text{as $r\downarrow0$}.\] 
Thus, boundedness of $\rho(t,r)$ as $r\downarrow0$ requires
\beq\label{rho_req}
	\kappa+\textstyle\frac{A-2\lambda}{1-\gamma+q}\geq0.
\eeq
We proceed to show that, under the integrability constraints (I)-(III), this last inequality 
together with \eq{alfa} yield $\alpha=0$. First, recall from \eq{q} that $q=\frac{2\gamma\alpha}{\kappa+n}$;
in particular, \eq{alfa} and (I) give $q\leq0$. As $\gamma>1$ we therefore have $1-\gamma+q<0$,
so that \eq{rho_req} may be rewritten as 
\[\kappa(1-\gamma+q)+A-2\lambda\leq0.\]
Substituting $q=\frac{2\gamma\alpha}{\kappa+n}$ and $A=2(1+\textstyle\frac{\alpha}{1+V_*})$, 
and rearranging, we obtain the equivalent inequality
\beq\label{rho_req2}
	(\textstyle\frac{\gamma\kappa}{\kappa+n}+\frac{1}{1+V_*})\alpha
	\leq (\lambda-1)+\frac{\kappa}{2}(\gamma-1)\equiv\alpha\gamma,
\eeq
cf.\ \eq{alpha}.
Now assume for contradiction that $\alpha<0$. According \eq{alpha} we then have
\beq\label{rho_req3}
	\lambda<1-\textstyle\frac{\kappa}{2}(\gamma-1),
\eeq
while \eq{rho_req2} gives
\beq\label{rho_req4}
	\textstyle\frac{1}{1+V_*}\geq \frac{\gamma n}{\kappa+n}.
\eeq
According to (I) and \eq{1+V*}, both  $\kappa+n$ and $1+V_*$ 
are strictly positive quantities. Using this in \eq{rho_req4}, and also
the expression \eq{V_*} for $V_*$, we rewrite \eq{rho_req4} as
\[\lambda\geq 1+\textstyle\frac{n}{2}(\gamma-1),\]
which together with \eq{rho_req3} gives 
\[1+\textstyle\frac{n}{2}(\gamma-1)<1-\textstyle\frac{\kappa}{2}(\gamma-1)\qquad
\text{or}\qquad (\kappa+n)(\gamma-1)<0.\]
As $\gamma>1$ we obtain $\kappa+n<0$, in contradiction to (I). Therefore,  
$\alpha$ cannot be strictly negative, and \eq{alfa} gives $\alpha=0$.

Opting to keep $\lambda$ as a parameter, it follows from the expression 
\eq{alpha} that the similarity parameter $\kappa$ 
takes the isentropic value $\bar\kappa$ in \eq{isentr_kappa}.
We have established the following:
%%%%%%%%%%%%%%%%%%%%%%%%%%%%%%%%%%%
\begin{lemma}\label{kappa_value}
	Consider a radial, self-similar solution of the form \eq{sim_vars} to the 
	Euler system \eq{m_eul}-\eq{mom_eul}-\eq{ener_eul}. 
	Assume that the solution satisfies the constraints (C1)-(C3). 
	Then the similarity parameter $\kappa$ must necessarily take the value
	\beq\label{isentr_kappa2}
		\bar\kappa=\bar\kappa(\gamma,\lambda):=-\textstyle\frac{2(\lambda-1)}{\gamma-1},
	\eeq
	and the flow is therefore globally isentropic (cf. \eq{CR}).
\end{lemma}
%%%%%%%%%%%%%%%%%%%%%%%%%%%%%%%%%%%
\begin{remark}
	Conversely, it is evident from the calculations above that with $n=2$ or $n=3$, 
	$\gamma>1$, and $\lambda>1$ given, 
	the choice $\kappa=\bar\kappa(\gamma,\lambda)$ guarantees that 
	the constraint (C2) is satisfied.
\end{remark}
%%%%%%%%%%%%%%%%%%%%%%%%%%%%%%%%%%%

Before proceeding we record some consequences of Lemma \ref{kappa_value}. 
First, observe that 
\eq{isentr_kappa2} is equivalent to $\alpha=0$, so that \eq{ab} gives $A>0$,
which is the condition guaranteeing that $P_{\pm\infty}$ are saddles.
In particular, the isentropic $\kappa$-value in \eq{isentr_kappa2} is consistent 
with the second part of constraint (C3).

Next, consider the inequalities (I)-(III). A calculation
shows that, with $\kappa=\bar\kappa$, they reduce to the single constraint
\beq\label{1st_lam_constr}
	\lambda<\tilde\lambda_n(\gamma):=1+\textstyle\frac{n}{2}(1-\frac{1}{\gamma}).
\eeq
Also, as detailed in Section \ref{P1_P3}, amplitude blowup in $u(t,r)$ and $c(t,r)$ 
at time $t=0$ requires $\lambda>1$. It follows that $\bar\kappa<0$; in particular, 
it follows from \eq{sim_vars}${}_2$ that also the density $\rho(t,r)$ suffers blowup
at collapse. We note that \eq{pressure1}-\eq{sound_speed} implies that the same 
applies to the temperature $\theta\propto c^2$ and pressure $p\propto \rho\theta$.

From here on the following  assumptions are in force:
\beq\label{assumpns_1}
	n=2\quad\text{or}\quad n=3, \qquad \gamma>1,
	\qquad\kappa=\bar\kappa(\gamma,\lambda),\qquad 
	1<\lambda<\tilde\lambda_n(\gamma).
\eeq
Note that, with the isentropic $\kappa$-value, we have (cf.\ \eq{V_*})
\beq\label{isntr_V_star}
	V_*=-\textstyle\frac{2(\lambda-1)}{n(\gamma-1)}=-\textstyle\frac{2\mu}{n\eps}<0.
\eeq
Summing up the analysis above we have:
%%%%%%%%%%%%%%%%%%%%%%%%%%%%%%%%%%%%%%
\begin{proposition}\label{constraints_C1_C_2_and _2nd_C3}
	Assume \eq{assumpns_1} holds. Then the constraints (C1) and (C2),
	and also the second part of constraint (C3) concerning $P_{\pm\infty}$, are all met. 
\end{proposition}
%%%%%%%%%%%%%%%%%%%%%%%%%%%%%%%%%%%%%%
As we shall see below, the $\lambda$-range in \eq{assumpns_1} will be delimited 
further by the requirement that the critical points $P_8$ and $P_9$ are present as nodes 
(i.e., the first part of constraint (C3)).

%%%%%%%%%%%%%%%%%%%%%%%%%%%%%%%%%%%%%%%%
%%%%%%%%%%%%%%%%%%%%%%%%%%%%%%%%%%%%%%%%
\section{Presence, location, and types of $P_4$-$P_9$}\label{presence_type}
%%%%%%%%%%%%%%%%%%%%%%%%%%%%%%%%%%%%%%%%
%%%%%%%%%%%%%%%%%%%%%%%%%%%%%%%%%%%%%%%%
With the assumptions \eq{assumpns_1} in force, we proceed to analyze when 
the critical points $P_8$ and $P_9$ are present, and if so, when they are nodes.

%%%%%%%%%%%%%%%%%%%%%%%%%%%%%%%%%%%%%%%%
\subsection{Preliminaries}\label{prelimns}
%%%%%%%%%%%%%%%%%%%%%%%%%%%%%%%%%%%%%%%%
Necessary conditions for the presence of 
$P_8$ and $P_9$ were listed in (A)-(B) in Section \ref{P4_P9}. For $\gamma\neq 3$ 
they involve the two functions $\hat\lambda_n(\gamma)$ and $\check\lambda_n(\gamma)$,
which satisfy  $\hat\lambda_n(\gamma)<\check\lambda_n(\gamma)$ 
(cf.\ \eq{lambda_max}-\eq{lambda_min}).
To avoid dealing with a number of sub-cases, we choose the lower $\lambda$-value
and require
\beq\label{lambda_hat}
	1<\lambda< \hat\lambda_n(\gamma)
	=1+\textstyle\frac{(n-1)(\gamma-1)}{(\gamma+1)+\sqrt{8(\gamma-1)}},
\eeq
for any $n$, $\gamma$ under consideration (including $\gamma=3$). 
Under this requirement we have $\mathfrak Q_n(\gamma,\lambda)>0$
for all $n$, $\gamma$, $\lambda$ under consideration. In particular, this 
implies that $P_6$, $P_7$, $P_8$, $P_9$, when present, are four distinct points.

Before proceeding we compare the constraint \eq{lambda_hat} with the standing 
assumption in \eq{1st_lam_constr}.
It is immediate to verify that for $n=2$, $\hat\lambda_2(\gamma)<\tilde\lambda_2(\gamma)$
for all $\gamma>1$. On the other hand, for $n=3$ and $\gamma>1$, 
$\hat\lambda_3(\gamma)<\tilde\lambda_3(\gamma)$ holds if and only if 
$1<\gamma<\gamma_+:=3(13+4\sqrt{10})\approx 76.9$.
Defining 
\beq\label{lambda_star_n=2}
	\lambda^*_2(\gamma):=\hat\lambda_2(\gamma) 
\eeq
and
\beq\label{lambda_star_n=3}
	\lambda^*_3(\gamma):=
	\left\{
	\begin{array}{ll}
		\hat\lambda_3(\gamma) & \text{when $n=3$ and $1<\gamma<\gamma_+$}\\
		\tilde\lambda_3(\gamma)	& \text{when $n=3$ and $\gamma\geq\gamma_+$},
	\end{array}
	\right.
\eeq
we therefore have the following:
%%%%%%%%%%%%%%%%%%%%%%%%%%%%%%%%%%%
\begin{lemma}\label{lambda_star_n_constr}
        For $n=2$ or $n=3$, $\gamma>1$, and $\kappa=\bar\kappa(\gamma,\lambda)$,
        the conditions (C1) and (C2) are met and the radicand 
        $\mathfrak Q_n(\gamma,\lambda)$ in \eq{isntr_V_pm} (determining $V_\pm$) 
        is strictly positive whenever
        \beq\label{2nd_lam_constr}
        		1<\lambda<\lambda^*_n(\gamma).
        \eeq
\end{lemma}
%%%%%%%%%%%%%%%%%%%%%%%%%%%%%%%%%%%
We now assume \eq{2nd_lam_constr} is met.
The following analysis will show that, for $n$ and $\gamma$ fixed, the $\lambda$-range 
needs to be restricted further in order to meet the constraint (C3).
For this we begin by deriving conditions for having the pair of points $P_4,P_5$
located to the left of the pair of points $P_8,P_9$ in the $(V,C)$-plane, i.e., 
we seek conditions guaranteeing that
\beq\label{P_4_left_of_P_8}
	V_4<V_+,
\eeq
cf.\ \eq{V_4_C_4} and \eq{P_6-P_9}.
The inequality \eq{P_4_left_of_P_8} is relevant for two reasons: 
With $n$ and $\gamma$ fixed, and under suitable conditions on $\lambda$, 
it will be used to argue that $P_4$-$P_9$ are all present, and that $P_8,P_9$ 
are nodes.

Substituting from \eq{V_4_C_4} and \eq{isntr_V_pm}-\eq{aQ} and 
rearranging show that \eq{P_4_left_of_P_8} is equivalent to
\beq\label{q_n_A_n}
	\sqrt{\mathfrak q_n(\gamma,\lambda)}>A_n(\gamma,\lambda),
\eeq
where 
\beq\label{q_n}
	\mathfrak q_n(\gamma,\lambda):=m^2\eps^2\mathfrak Q_n(\gamma,\lambda)
	=(\eps-2)^2\mu^2
	-2m\eps(\eps+2)\mu+m^2\eps^2,
\eeq
and
\[A_n(\gamma,\lambda):=\textstyle\frac{m\eps}{2+n\eps}
[n\eps-2(2\mu+1)]-(\eps-2)\mu.\]
To analyze \eq{q_n_A_n} it is convenient to consider the cases $n=2$ and $n=3$
separately. The calculations are elementary but somewhat lengthy, and we omit 
some of the details.

%%%%%%%%%%%%%%%%%%%%%%%%%%%%%%%%%%%%%%%%
\subsubsection{$V_4<V_+$ for $n=2$}\label{sign_W_9_n=2}
%%%%%%%%%%%%%%%%%%%%%%%%%%%%%%%%%%%%%%%%
With $n=2$, \eq{q_n_A_n} takes the form
\beq\label{W_9_n=2_ineq}
	\sqrt{\mathfrak q_2(\gamma,\lambda)}
	>A_2(\gamma,\lambda)\equiv\textstyle\frac{(\gamma+1)}{\gamma}(2-\gamma)
	[\lambda-\frac{2\gamma}{\gamma+1}].
\eeq
A direct calculation using \eq{lambda_hat} and \eq{lambda_star_n=3} verifies that 
$\lambda^*_2(\gamma)<\frac{2\gamma}{\gamma+1}$ for all $\gamma>1$.
Therefore, under the current assumption that $1<\lambda<\lambda^*_2(\gamma)$, the 
square-bracketed term on the right-hand side of the inequality in
\eq{W_9_n=2_ineq} is negative. 
In particular, for  $1<\gamma\leq 2$, $A_2(\gamma,\lambda)$ 
is non-positive, and thus \eq{q_n_A_n} holds, i.e., $V_4<V_+$ in this case.

For $\gamma>2$ the situation is more involved: $A_2(\gamma,\lambda)$ is then strictly 
positive for $1<\lambda<\lambda^*_2(\gamma)$, and we need to evaluate the inequality 
in \eq{W_9_n=2_ineq}. Squaring both sides and rearranging we get that $V_4<V_+$ now holds 
if and only if 
\beq\label{W_9_n=2_ineq_gamma_gtr_2}
	(\gamma^2-2\gamma-1)(\lambda-1)^2+2(\gamma^2-1)(\lambda-1)-(\gamma-1)^2<0.
\eeq
A calculation verifies that that \eq{W_9_n=2_ineq_gamma_gtr_2} holds
whenever $\gamma>2$ and $1<\lambda<\frac{\gamma\sqrt 2 }{\gamma+\sqrt 2 -1}$.
A further calculation shows that the upper limit in the last expression satisfies
$\frac{\gamma\sqrt 2 }{\gamma+\sqrt 2 -1}<\lambda_2^*(\gamma)$ for all
$\gamma>2$. We therefore delimit the $\lambda$-range
further by setting 
\beq\label{lambda_0_n=2}
	\lambda^\circ_2(\gamma):=
	\left\{
	\begin{array}{ll}
		\hat\lambda_2(\gamma)
		=1+\textstyle\frac{\gamma-1}{(\gamma+1)+\sqrt{8(\gamma-1)}} & 1<\gamma\leq 2\\
		\frac{\gamma\sqrt 2 }{\gamma+\sqrt 2 -1} & \gamma> 2,
	\end{array}
	\right.
\eeq
and we require from now on that
\beq\label{3rd_lam_constr_n=2}
	1<\lambda<\lambda^\circ_2(\gamma)
\eeq
when $n=2$ and $\gamma>1$.

%%%%%%%%%%%%%%%%%%%%%%%%%%%%%%%%%%%%%%%%
\subsubsection{$V_4<V_+$ for $n=3$}\label{sign_W_9_n=3}
%%%%%%%%%%%%%%%%%%%%%%%%%%%%%%%%%%%%%%%%
The analysis is similar to the case $n=2$. First, with $n=3$, \eq{q_n_A_n} takes the form
\beq\label{W_9_n=3_ineq}
	\sqrt{\mathfrak q_3(\gamma,\lambda)}
	>A_3(\gamma,\lambda)\equiv\textstyle\frac{\gamma+1}{3\gamma-1}(5-3\gamma)
	[\lambda-\frac{3\gamma-1}{\gamma+1}].
\eeq
A direct calculation using \eq{lambda_hat} and \eq{lambda_star_n=3} verifies that 
$\lambda^*_3(\gamma)<\frac{3\gamma-1}{\gamma+1}$ for all $\gamma>1$.
Therefore, under the current assumption that $1<\lambda<\lambda^*_3(\gamma)$, the 
square-bracketed term on the right-hand side of the inequality in
\eq{W_9_n=3_ineq} is negative. 
In particular, for  $1<\gamma\leq \frac{5}{3}$, $A_3(\gamma,\lambda)$ 
is non-positive, and thus \eq{q_n_A_n} holds, i.e., $V_4<V_+$ in this case.

If instead $\gamma>\frac{5}{3}$, then $A_3(\gamma,\lambda)$ is  strictly 
positive for $1<\lambda<\lambda^*_3(\gamma)$, and we need to evaluate the inequality 
in \eq{W_9_n=3_ineq}. Squaring both sides and rearranging we get that $V_4<V_+$ now holds 
if and only if 
\beq\label{W_9_n=3_ineq_gamma_gtr_5_over_3}
	(3\gamma^2-6\gamma-1)(\lambda-1)^2+6(\gamma^2-1)(\lambda-1)-6(\gamma-1)^2<0.
\eeq
A calculation verifies that that \eq{W_9_n=3_ineq_gamma_gtr_5_over_3} holds
whenever $\gamma>\frac{5}{3}$ and $1<\lambda<\frac{3\gamma-1 }{\sqrt 3(\gamma-1)+2}$.
A further calculation shows that the upper limit in the last expression satisfies
$\frac{3\gamma-1 }{\sqrt 3(\gamma-1)+2}<\lambda_3^*(\gamma)$ for all
$\gamma>\frac{5}{3}$. We therefore delimit the  $\lambda$-range 
further  by setting 
\beq\label{lambda_0_n=3}
	\lambda^\circ_3(\gamma):=
	\left\{
	\begin{array}{ll}
		\lambda_3^*(\gamma)
		=1+\textstyle\frac{2(\gamma-1)}{(\gamma+1)+\sqrt{8(\gamma-1)}} & 1<\gamma\leq \frac 5 3\\
		\frac{3\gamma-1 }{\sqrt 3(\gamma-1)+2} & \gamma> \frac 5 3,
	\end{array}
	\right.
\eeq
and we require from now on that
\beq\label{3rd_lam_constr_n=3}
	1<\lambda<\lambda^\circ_3(\gamma)
\eeq
when $n=3$ and $\gamma>1$.

Summarizing the analysis above, we have:
%%%%%%%%%%%%%%%%%%%%%%%%%%%%%%%%%%%%%%%%
\begin{lemma}\label{V_4_less_V_+}
	Assume that $n=2$ or $n=3$, $\gamma>1$,  $\kappa=\bar\kappa(\gamma,\lambda)$, 
	and $1<\lambda<\lambda^\circ_n(\gamma)$. Then,
 	\beq\label{V_4_less_V_+_eqn}
		V_4<V_+.
	\eeq
\end{lemma}

%%%%%%%%%%%%%%%%%%%%%%%%%%%%%%%%%
\begin{remark}\label{useful_rmk}
	For later reference we record that, with $n=2$ or $n=3$, the function 
	$\gamma\mapsto\lambda^\circ_n(\gamma)$ 
	is continuous, strictly increasing, and bounded above by $\sqrt{n}<2$.
	Also, the definitions of $\lambda^\circ_n(\gamma)$, $\hat\lambda_n(\gamma)$ 
	and $\tilde\lambda_n(\gamma)$ give
	\beq\label{hat_circ1}
		1<\lambda^\circ_n(\gamma)\leq\hat\lambda_n(\gamma)
	\eeq
	and
	\beq\label{hat_circ2}
		1<\lambda^\circ_n(\gamma)\leq\tilde\lambda_n(\gamma)
	\eeq
	for all $\gamma>1$.
\end{remark}
%%%%%%%%%%%%%%%%%%%%%%%%%%%%%%%%%

%%%%%%%%%%%%%%%%%%%%%%%%%%%%%%%%%%%%%%%%
\subsection{Presence and locations of $P_4$-$P_9$}\label{P_8_9_locns}
%%%%%%%%%%%%%%%%%%%%%%%%%%%%%%%%%%%%%%%%
The following lemma gives more precise information about the presence and location 
of the critical points $P_4$-$P_9$. We note that this is the first place where 
we exploit the freedom of choosing $\lambda>1$ sufficiently close to $1$; we will 
do so also in later results.
%%%%%%%%%%%%%%%%%%%%%%%%%%%%%%%%%%%%%%%%
\begin{lemma}\label{pres_locn}
	Assume that $n=2$ or $n=3$, $\gamma>1$,  $\kappa=\bar\kappa(\gamma,\lambda)$, 
	and $1<\lambda<\lambda^\circ_n(\gamma)$. Then,
 	\beq\label{ineqs_1}
		-1<V_-,\, V_4<V_+<V_*<0,
	\eeq
	where $V_\pm$, $V_4$, $V_*$ are given in \eq{isntr_V_pm},
	\eq{V_4_C_4}, \eq{isntr_V_star}, respectively.
	It follows that all of the critical points $P_4$-$P_9$ are present with
	$P_6,P_8\in L_+$ and $P_7,P_9\in L_-$. Furthermore, we have that 
	\beq\label{P_4_locn}
		V_-<V_4 \qquad\text{and}\qquad C_4<1+V_4
	\eeq 
	whenever $\lambda>1$ is sufficiently close to $1$ (depending on $\gamma$ and $n$). 
	In particular, for $\lambda\gtrapprox1$, $P_4$ ($P_5$) is located below (above) $L_+$ ($L_-$)
	in the upper (lower) half-plane.
\end{lemma}
%%%%%%%%%%%%%%%%%%%%%%%%%%%%%%%%%%%%%%%%
\begin{proof} 
	We consider each of the inequalities in \eq{ineqs_1} in turn, from left to right.
	\begin{itemize}
		\item $-1<V_-$: According to \eq{isntr_V_pm} this amounts to
		$\sqrt{\mathcal Q_n}<\mathfrak a_n+2$. We claim that $\mathfrak a_n+2>0$,
		i.e., $(\eps-2)\mu+m\eps>0$. This is immediate
		when $\eps\geq 2$, while for $\eps< 2$ it holds provided $\mu<\frac{m\eps}{2-\eps}$,
		which is an easy consequence of $\lambda<\hat\lambda_n(\gamma)$. According 
		to \eq{hat_circ1} and the assumptions of the lemma, we therefore have 
		$\mathfrak a_n+2>0$ for all cases under consideration. Thus,
		$\sqrt{\mathcal Q_n}<\mathfrak a_n+2$ holds if and only if 
		$\mathfrak Q_n<(\mathfrak a_n+2)^2$. Expanding the right-hand side, 
		using the expressions from \eq{aQ}, and simplifying, yield the equivalent 
		inequality $-(\eps+2)\mu<m\eps$, which is trivially satisfied since $\eps,\mu,m>0$.
		\item $-1<V_4$: According to \eq{V_4_C_4} this amounts to 
		$\lambda<1+\frac{n}{2}(\gamma-1)$. Since 
		$\tilde\lambda_n(\gamma)<1+\frac{n}{2}(\gamma-1)$, this is an immediate 
		consequence of \eq{hat_circ2} and the assumptions of the lemma.
		\item $V_-<V_+$: this is immediate from the definitions of $V_\pm$ 
		(cf.\ \eq{isntr_V_pm}). 
		\item $V_4<V_+$: This is the content of Lemma \ref{V_4_less_V_+}.
		\item $V_+<V_*$: According to \eq{isntr_V_pm}-\eq{aQ} and \eq{isntr_V_star}, 
		this amounts to 
			\beq\label{intermed2}
				mn\eps\sqrt{\mathfrak Q_n} < mn\eps-(n(\eps+2)-4)\mu.
			\eeq
    		As $n(\eps+2)-4>0$ for all cass under consideration, the 
    		right-hand side in \eq{intermed2} is positive provided 
    		$\mu<\frac{mn\eps}{n(\eps+2)-4}$. It is straightforward 
    		to verify that the latter is a consequence of $\lambda<\hat\lambda_n(\gamma)$,
    		which holds according to \eq{hat_circ1} and the assumptions in the lemma. 
    		Squaring both sides of \eq{intermed2}, and simplifying, we therefore get that \eq{intermed2}
		is equivalent to $(2-n\eps)\mu<n\eps$. The latter inequality
    		is trivially satisfied when $2-n\eps\leq 0$, while for $2-n\eps>0$ it is
    		equivalent to 
    		\[\lambda<1+\textstyle\frac{n(\gamma-1)}{2-n(\gamma-1)}.\]
    		Again, this is an easy consequence of $\lambda<\hat\lambda_n(\gamma)$.
		\item $V_*<0$: This follows directly from \eq{isntr_V_star} and the assumptions
		in the lemma.
	\end{itemize}
	Next, recall from \eq{g} that $P_4$-$P_9$ are present provided $g(V_4)>0$ and $g(V_\pm)>0$.
	The latter inequalities follow from the inequalities in \eq{ineqs_1} and the 
	expression for $g(V)$ in \eq{g}. Also, as $V_+>V_->-1$, and $P_6,P_8$ ($P_7,P_9$) are located in 
	the upper (lower) half-plane, it follows from \eq{on_L_+-} that 
	$P_6,P_8\in L_+$ ($P_7,P_9\in L_-$).
	
	Finally, consider the inequalities in \eq{P_4_locn}, with $\gamma>1$ and $n$ fixed.
	From the expressions for $V_-$ and $V_4$ 
	in \eq{isntr_V_pm} and \eq{V_4_C_4}, we get that the inequality $V_-<V_4$ 
	reduces to $-1<-\frac{2}{2+n(\gamma-1)}$, which is trivially satisfied.
	Also, according to \eq{g}-\eq{V_4_C_4} and \eq{ineqs_1},
	the inequality $C_4<1+V_4$ amounts to $V_4(\lambda+V_4)>n(V_4-V_*)(1+V_4)$.
	The limiting version at $\lambda=1$ is the inequality $V_4(1+V_4)>nV_4(1+V_4)$
	which is clearly satisfied for $V_4=V_4|_{\lambda=1}=-\frac{2}{n\eps+2}\in(-1,0)$.
	It follows by continuity that both inequalities in \eq{P_4_locn} are satisfied whenever 
	$\lambda>1$ is sufficiently close to $1$.
\end{proof}
%%%%%%%%%%%%%%%%%%%%%%%%%%
\begin{remark}
	A more detailed calculation reveals that for both $n=2$ and $n=3$, 
	the inequality $V_-<V_4$ fails for pairs $(\gamma,\lambda)$
	with $\gamma\gtrapprox1$ and $\lambda\lessapprox\hat\lambda_n(\gamma)$.
	\end{remark}
%%%%%%%%%%%%%%%%%%%%%%%%%%

We can now address the type of the critical points $P_8$ and $P_9$.

%%%%%%%%%%%%%%%%%%%%%%%%%%%%%%%%%%%%%%%%
\subsection{Nodality of $P_8,P_9$}\label{P_8_9_nodes}
%%%%%%%%%%%%%%%%%%%%%%%%%%%%%%%%%%%%%%%%
According to (C3) it remains to determine the values of $\lambda$ for which $P_8$
and $P_9$ are nodes. 
We start
by recording some terminology and results from standard ODE theory (following 
the notation in \cite{laz}). 
The linearization of \eq{CV_ode} at a critical point $P_c=(V_c,C_c)$ is 
\beq\label{linzd_ode}
	\textstyle\frac{d\tilde C}{d\tilde V}=\frac{F_V\tilde V+F_C\tilde C}{G_V\tilde V+G_C\tilde C}
\eeq
where $\tilde V=V-V_c$, $\tilde C=C-C_c$, and the partials $F_V, F_C,G_V,G_C$ are evaluated at $P_c$.
The Wronskian $W$ and discriminant $R^2$ for the ODE \eq{CV_ode} at $P_c$ are defined by
\beq\label{wronsk}
	W:=F_CG_V-F_VG_C
\eeq
and
\beq\label{R_sqrd}
	R^2:=(F_C-G_V)^2+4F_VG_C\equiv (F_C+G_V)^2-4W.
\eeq
In what follows, we write
$R$ for the positive square root of $R^2$ whenever the latter is postive.
Assuming at present that $R^2>0$, we define the quantities
\beq\label{L_12}
	L_{1,2}=\textstyle\frac{1}{2G_C}(F_C-G_V\pm R)
\eeq
and
\beq\label{E_12}
	E_{1,2}=\textstyle\frac{1}{2G_C}(F_C+G_V\pm R),
\eeq
with signs chosen so that 
\beq\label{Es}
	|E_1|<|E_2|.
\eeq
Note that the signs $\pm$ in \eq{L_12} and in \eq{E_12} agree.
We then have that integrals of \eq{linzd_ode} near $P_c$ 
approach one of the curves
\beq\label{gen_crit_point}
	(\tilde C-L_1\tilde V)^{E_1}=\text{constant}\times (\tilde C-L_2\tilde V)^{E_2}.
\eeq
Here, $L_1$ and $L_2$ are  
the {\em primary} and {\em secondary} slopes (or directions), respectively, at $P_c$.
We observe that
\beq\label{W_prod}
	W\equiv E_1E_2G_C^2.
\eeq 
Thus, under the assumption that $R^2>0$, the critical point 
$P_c$ is a node (i.e., $\sgn E_1=\sgn E_2$) when $W>0$, and it is a 
saddle (i.e., $\sgn E_1=-\sgn E_2$) when 
$W<0$. 

In the nodal case it follows from \eq{gen_crit_point} that all but two of the integrals 
of \eq{CV_ode} near $P_c$ 
approach $P_c$ with slope $L_1$. The exceptions are the two integrals  
approaching $P_c$ with slope $L_2$ (one from each side of the line $C-C_c=L_1(V-V_c)$).
Similarly, in the saddle case, exactly four integrals 
of \eq{CV_ode} near $P_c$ 
approach $P_c$, two of them with slope $L_1$ and two with slope $L_2$. 

To guarantee that $P_8$ and $P_9$ are nodes,
we require $R^2>0$ and $W>0$ at $P_8$ and $P_9$. 
Note that, according to \eq{symms}, $W_8\equiv W_9$ and $R^2_8\equiv R^2_9$.
It therefore suffices to consider one of these points, and we use $P_9$.
We begin with the sign of $W_9$; the more intricate condition $R^2>0$
is analyzed in Section \ref{R^2_9}.

%%%%%%%%%%%%%%%%%%%%%%%%%%%%
\subsubsection{The Wronskian $W_9(n,\gamma,\lambda)$}\label{W_9}
%%%%%%%%%%%%%%%%%%%%%%%%%%%%
An elegant argument of Lazarus (see p.\ 323 in \cite{laz}) 
shows that the Wronskian $W_9$ factors as follows,
\beq\label{W_9_eqn}
	W_9=W_9(n,\gamma,\lambda)=KC_9^2(V_9-V_4)(V_9-V_7),
\eeq
where the positive constant $K$ is given by
\beq\label{K}
	K=m(n\eps+2).
\eeq
Recalling from Section \ref{P4_P9} that $V_9=V_+>V_-=V_7$,
we get from \eq{W_9_eqn} that $W_9>0$ if and only if $V_+>V_4$,
which holds under the assumptions in Lemma \ref{pres_locn}. We therefore have:
%%%%%%%%%%%%%%%%%%%%%%%%%%%%%%%%%%%%%
\begin{lemma}\label{W_9_pos}
	Assume that $n=2$ or $n=3$, $\gamma>1$,  $\kappa=\bar\kappa(\gamma,\lambda)$, 
	and $1<\lambda<\lambda^\circ_n(\gamma)$. 
	Assume in addition that the discriminant $R_9^2(n,\gamma,\lambda)$ is 
	strictly positive. Then the Wronskian at $P_9$ satisfies $W_9(n,\gamma,\lambda)>0$.
\end{lemma}
%%%%%%%%%%%%%%%%%%%%%%%%%%%%%%%%%%%%%

%%%%%%%%%%%%%%%%%%%%%%%%%%%%%%%%%%%%%%%%
\subsubsection{The discriminant $R^2_9(n,\gamma,\lambda)$}\label{R^2_9}
%%%%%%%%%%%%%%%%%%%%%%%%%%%%%%%%%%%%%%%%
As above we assume $n=2$ or $n=3$, $\gamma>1$,  $\kappa=\bar\kappa(\gamma,\lambda)$, 
and $1<\lambda<\lambda^\circ_n(\gamma)$, and the goal is to provide sufficient conditions on 
$\lambda$ to have $R^2_9>0$.

To simplify the notation we drop the subscript `$9$' in most cases, evaluation at $P_9=(V_9,C_9)$ 
being understood. Since the functions $D(V,C)$, $F(V,C)$, $G(V,C)$ 
all vanish at $P_9$, we have
\begin{align}
	C&=-(1+V)\label{at_P_9_1}\\
	C^2&=k_1(1+V)^2-k_2(1+V)+k_3\label{at_P_9_2}\\
	C^2&=\textstyle\frac{V(1+V)(\lambda+V)}{n(V-V_*)}.\label{at_P_9_3}
\end{align}
By using these we have the following expressions for the partials of $F$ and $G$ 
at $P_9$:
\begin{align}
	F_C&=2C^2\label{F_C}\\
	F_V&=C[k_2-2k_1(1+V)]\label{F_V}\\
	G_C&=2nC(V-V_*)\equiv-2V(\lambda+V)\label{G_C}\\
	G_V&=nC^2-[(1+V)(\lambda+V)+V(\lambda+V)+V(1+V)]\equiv
	C[2V+\lambda-n(1+V_*)].\label{G_V}
\end{align}
Recalling that $R^2=(F_C+G_V)^2-4W$, and the expression
\eq{W_9_eqn} for $W$, a direct calculation shows that $R^2_9=R^2_9(n,\gamma,\lambda)>0$
holds if and only if
\begin{align*}
	&2[(n\eps^2+2m\eps-4)\mu-m\eps(n\eps-2)]\sqrt{\mathfrak q_n(\gamma,\lambda)}\nn\\
	&\qquad\qquad <[-2n\eps^3+(9n-5)\eps^2-4(n-3)\eps+4(n-5)]\mu^2\nn\\
	&\qquad\qquad\quad +2m\eps(\eps+2)(2n\eps-n+3)\mu
	-m\eps^2[2mn\eps-(n^2-2n+5)].
\end{align*}	
Thus,
\begin{itemize}
	\item for $n=2$, $R^2_9>0$ if and only if
	\begin{align}
		4(\gamma-2)[(\gamma+1)\mu-(\gamma-1)]\sqrt{\mathfrak q_2(\gamma,\lambda)}
		&<(-4\gamma^3+25\gamma^2-34\gamma+1)\mu^2\nn\\
		&\quad+2(\gamma^2-1)(4\gamma-3)\mu-(4\gamma-9)(\gamma-1)^2,
		\label{R^2_pos_n=2}
	\end{align}
	where 
	\[\mathfrak q_2(\gamma,\lambda)=(\gamma-3)^2\mu^2-2(\gamma^2-1)\mu+(\gamma-1)^2;\]
\end{itemize}
and, 
\begin{itemize}
	\item for $n=3$, $R^2_9>0$ if and only if
	\begin{align}
		(3\gamma-5)[(\gamma+1)\mu-2(\gamma-1)]\sqrt{\mathfrak q_3(\gamma,\lambda)}
		&<-(3\gamma-5)(\gamma^2-5\gamma+2)\mu^2\nn\\
		&\quad+12(\gamma-1)^2(\gamma+1)\mu-4(3\gamma-5)(\gamma-1)^2,
		\label{R^2_pos_n=3}
	\end{align}
	where 
	\[\mathfrak q_3(\gamma,\lambda)=(\gamma-3)^2\mu^2-4(\gamma^2-1)\mu+4(\gamma-1)^2.\]
\end{itemize}
While numerical plots indicate that \eq{R^2_pos_n=2} (\eq{R^2_pos_n=3}, respectively) 
holds whenever $\gamma>1$ and $1<\lambda<\lambda^\circ_2(\gamma)$ 
($1<\lambda<\lambda^\circ_3(\gamma)$, respectively), we have not been able to prove this.
However, the following weaker assertion will suffice for our needs.

%%%%%%%%%%%%%%%%%%%%%%%%%%%%%%%%%%%%
\begin{lemma}\label{R^2_pos_suff} 
	Assume that $n=2$ or $n=3$, $\gamma>1$,  $\kappa=\bar\kappa(\gamma,\lambda)$, 
	and $1<\lambda<\lambda^\circ_n(\gamma)$. 
	Then $R^2_9(n,\gamma,\lambda)>0$ whenever $\lambda>1$ is sufficiently close to $1$
	(depending on $\gamma$ and $n$).
\end{lemma}
%%%%%%%%%%%%%%%%%%%%%%%%%%%%%%%%%%%%
\begin{proof}
	For $n=2$, this follows directly by observing that along $\lambda\equiv 1$, 
	\eq{R^2_pos_n=2} reduces to $-4(\gamma-2)(\gamma-1)^2<-(4\gamma-9)(\gamma-1)^2$,
	which is trivially satisfied for all $\gamma>1$, and the claim follows by continuity.
	
	On the other hand, when $n=3$, the two sides of \eq{R^2_pos_n=3} agree
	along $\lambda\equiv 1$. Letting $\mathfrak l(\gamma,\lambda)$ and 
	$\mathfrak r(\gamma,\lambda)$ denote the left-hand and right-hand sides
	in \eq{R^2_pos_n=3}, respectively, it thus suffices to verify the strict inequality 
	\[\left.\textstyle\frac{\partial \mathfrak l}{\partial \lambda}\right|_{\lambda=1}
	<\left.\textstyle\frac{\partial \mathfrak r}{\partial \lambda}\right|_{\lambda=1}.\]
	A calculation shows that the latter inequality amounts to
	$4(3\gamma-5)(\gamma^2-1)<12(\gamma-1)(\gamma^2-1)$, 
	which is trivially satisfied for all $\gamma>1$, and again the claim follows by continuity.
\end{proof}
%%%%%%%%%%%%%%%%%%%%%%%%%%%%%%%%%%%%
\begin{definition}\label{relevant}
	For $n=2$ or $n=3$, $\gamma>1$, and $\kappa=\bar\kappa(\gamma,\lambda)$, 
	the pair $(\gamma,\lambda)$ is called {\em relevant} provided
	$1<\lambda<\lambda^\circ_n(\gamma)$, and 
	$R^2_9(n,\gamma,\lambda)>0$. 
\end{definition}
%%%%%%%%%%%%%%%%%%%%%%%%%%%%%%%%%%%%
Note that, according to Lemma \ref{W_9_pos},
$W_9(n,\gamma,\lambda)>0$ whenever $(\gamma,\lambda)$ is relevant.
By combining Lemma \ref{R^2_pos_suff}  and Lemma \ref{W_9_pos} with 
Proposition \ref{constraints_C1_C_2_and _2nd_C3}, we obtain:
%%%%%%%%%%%%%%%%%%%%%%%%%%%%%%%%%%%%
\begin{proposition}\label{rel_prop}
	Assume $n=2$ or $n=3$, $\gamma>1$, and $\kappa=\bar\kappa(\gamma,\lambda)$.
	Then the pair $(\gamma,\lambda)$ is relevant whenever $\lambda>1$ is sufficiently 
	close to $1$, and in this case the constraints (C1), (C2), (C3) are all satisfied.
	In particular, $P_8$ and $P_9$ are nodes, and $P_{\pm\infty}$ are saddles.
\end{proposition}
%%%%%%%%%%%%%%%%%%%%%%%%%%%%%%%%%%%%

%%%%%%%%%%%%%%%%%%%%%%%%%%%%%%%%%%%%
%%%%%%%%%%%%%%%%%%%%%%%%%%%%%%%%%%%%
\section{Further properties of critical points}\label{further_props}
%%%%%%%%%%%%%%%%%%%%%%%%%%%%%%%%%%%%
%%%%%%%%%%%%%%%%%%%%%%%%%%%%%%%%%%%%
To construct continuous similarity flows we need to
determine the primary and secondary slopes at $P_8$ and $P_9$
under the assumptions of Proposition \ref{rel_prop}.
We shall also need to compare these with the slopes of the zero-levels 
$\mathcal F=\{F=0\}$ and $\mathcal G=\{G=0\}$ at $P_8$ and $P_9$. 
In addition, it will be useful to determine the types of the  
two critical points $P_4$ and $P_5$ (located off the critical lines $L_\pm$).
Recalling the symmetries in \eq{symms}, in particular the invariance 
of $W$ and $R^2$ under $(V,C)\mapsto(V,-C)$, it suffices to analyze 
the critical points in the lower half-plane.

Throughout this section it is assumed in all calculations and statements 
that $(\gamma,\lambda)$ is relevant
according to Definition \ref{relevant}; in particular, $R^2_9>0$ and $W_9>0$.

%%%%%%%%%%%%%%%%%%%%%%%%%%%%%%%%
\subsection{Primary and secondary slopes at $P_9$}\label{prim_secnd_P_9}
%%%%%%%%%%%%%%%%%%%%%%%%%%%%%%%%
In the following analysis we suppress the subscript `9' in most cases, 
evaluation at $P_9$ being understood.
%%%%%%%%%%%%%%%%%%%%%%%%%%%%%%%%
\begin{lemma}\label{prim_sec_direcns}
	The primary and secondary slopes at $P_9$ are
	\beq\label{prim_sec_slopes}
		L_1=\textstyle\frac{1}{2G_C}(F_C-G_V- R)
		\qquad\text{and}\qquad
		L_2=\textstyle\frac{1}{2G_C}(F_C-G_V+ R),
	\eeq
	respectively.
\end{lemma}
%%%%%%%%%%%%%%%%%%%%%%%%%%%%%%%%
\begin{proof}
	The primary and secondary slopes are given by \eq{L_12} and \eq{E_12}.
	Recall that the signs $\pm$ in these equations agree. We claim that 
	\beq\label{F_C+G_V_pos}
		F_C+G_V>0.
	\eeq
	Assuming this, it follows from \eq{R_sqrd}, $W>0$, and \eq{Es},
	that $E_1$ is given by choosing the `$-$' sign in \eq{E_12}, and \eq{prim_sec_slopes}
	follows.
	
	It only remains to argue for \eq{F_C+G_V_pos}. Using the expressions in \eq{F_C}, 
	\eq{G_V}, together with \eq{at_P_9_1}, we have
	\beq\label{F_C+G_V}
		F_C+G_V=2C^2+C[2V+\lambda-n(1+V_*)]=C[\lambda-2-n(1+V_*)].
	\eeq
	Since $(\gamma,\lambda)$ is assumed relevant, \eq{1+V*} gives $1+V_*>0$,
	while the first part of Remark \ref{useful_rmk} gives $\lambda<2$. 
	As $C=C_9<0$, \eq{F_C+G_V} therefore gives $F_C+G_V>0$.
\end{proof}
The next result provides the relative positions near $P_9$ of the level sets 
$\mathcal F$, $\mathcal G$ (whose slopes are -$\textstyle\frac{F_V}{F_C}$
and $-\textstyle\frac{G_V}{G_C}$, respectively)
and the straight lines through $P_9$ with slopes $L_{1,2}$.
%%%%%%%%%%%%%%%%%%%%%%%%%%%%%%%%
\begin{lemma}\label{rel_locns_P9}
	At $P_9$ we have
	\beq\label{rel_locns}
		-\textstyle\frac{G_V}{G_C}<L_1<-\textstyle\frac{F_V}{F_C}
		<-1<0<L_2.
	\eeq
\end{lemma}
%%%%%%%%%%%%%%%%%%%%%%%%%%%%%%%%
\begin{proof}
	We first observe that \eq{G_C}${}_1$, Lemma \ref{pres_locn}, 
	and $C<0$ give
	\beq\label{G_C_pos}
		G_C=2nC(V-V_*)>0.
	\eeq
	Using the expression for $L_1$ in \eq{prim_sec_slopes} we therefore get that 
	the leftmost inequality in \eq{rel_locns} holds if and only if 
	$F_C+G_V>R$. We showed in the proof of Lemma \ref{prim_sec_direcns}
	that $F_C+G_V>0$, and we have $R>0$ by assumption. The  
	inequality $F_C+G_V>R$ is therefore a direct consequence of \eq{R_sqrd}${}_2$ and $W>0$.
	
	Next consider the inequality $-\frac{F_V}{F_C}<-1$, which, according to \eq{F_C}, 
	is equivalent to $F_V>F_C$. Substituting the expressions in \eq{F_C}-\eq{F_V},
	and using \eq{at_P_9_1}, we obtain the equivalent inequality $k_2<2(k_1-1)(1+V)$.
	Using the expressions for $k_1,k_2$ in \eq{ks}, together with the expression for $V_9=V_+$ given
	by \eq{isntr_V_pm}-\eq{aQ}, the last inequality reduces to 
	$m(\gamma-1)\sqrt{\mathfrak Q_n}>0$, which trivially holds.
	
	Next, using that $F_C>0$, $G_C>0$ (by \eq{F_C} and \eq{G_C_pos})
	and \eq{prim_sec_slopes}${}_1$, we have that $L_1<-\textstyle\frac{F_V}{F_C}$
	holds if and only if 
	\beq\label{2nd_ineq}
		F_C(F_C-G_V)+2F_VG_C<F_CR.
	\eeq
	This inequality is trivially satisfied if the left-hand side is non-positive.
	If not we square both sides of \eq{2nd_ineq} and use the first expression for 
	$R^2$ in \eq{R_sqrd} to simplify, and obtain the 
	equivalent inequality 
	\beq\label{intermed3}
		F_VG_C(F_CG_V-F_VG_C)\equiv F_VG_CW>0.
	\eeq
	We just established $F_V>F_C>0$ and $G_C>0$ above, while $W>0$ by assumption.
	Thus, \eq{intermed3} holds, and this verifies $L_1<-\textstyle\frac{F_V}{F_C}$ for 
	all cases under consideration.
	
	Finally, according to \eq{prim_sec_slopes}${}_2$ and \eq{G_C_pos}, the 
	inequality $L_2>0$ reduces to $F_C-G_V+R>0$ which, by \eq{R_sqrd}${}_1$,
	amounts to
	\[F_C-G_V+\sqrt{(F_C-G_V)^2+4F_VG_C}>0.\] 
	As $F_V>0$, $G_C>0$, this last inequality holds, establishing the rightmost 
	inequality in \eq{rel_locns}.
\end{proof}
The situation in the lower half-plane is given schematically in Figure \ref{Figure_1}.
Observe that Figure \ref{Figure_1} does not include certain parts of the zero-levels $\{F=0\}$ 
and $\{G=0\}$ (roughly those to the left of $P_5$; these are not relevant for what follows). 
Also, while Lemma \ref{rel_locns_P9}
shows that the behavior near $P_9$ is accurately depicted, we omit the details for verifying 
that Figure \ref{Figure_1} provides the correct global situation to the right of $P_5$. 
(This requires a straightforward analysis of the polynomial functions $F$ and $G$.) 
However, we record two properties that are useful for determining the flow direction of the original 
similarity ODEs \eq{V_sim2}-\eq{C_sim2}:
(p1): Both $\{F=0\}$ and $\{G=0\}$ tend to infinity with asymptotically constant slopes in the 
4th quadrant, with $\{G=0\}$ everywhere located above $\{F=0\}$; and (p2):
With $C<0$ fixed and $V\to+\infty$, it follows from \eq{F}, \eq{G} that $G(V,C)<0<F(V,C)$.

%%%%%%%%%%%%%%%%%%%%%%%%%%%%%%%%%%%%%%%%%%%%
%%%%%%%%%%%%%%%%%%%
%	FIGURE 
%%%%%%%%%%%%%%%%%%%
\begin{figure}
	\centering
	\includegraphics[width=11cm,height=11cm]{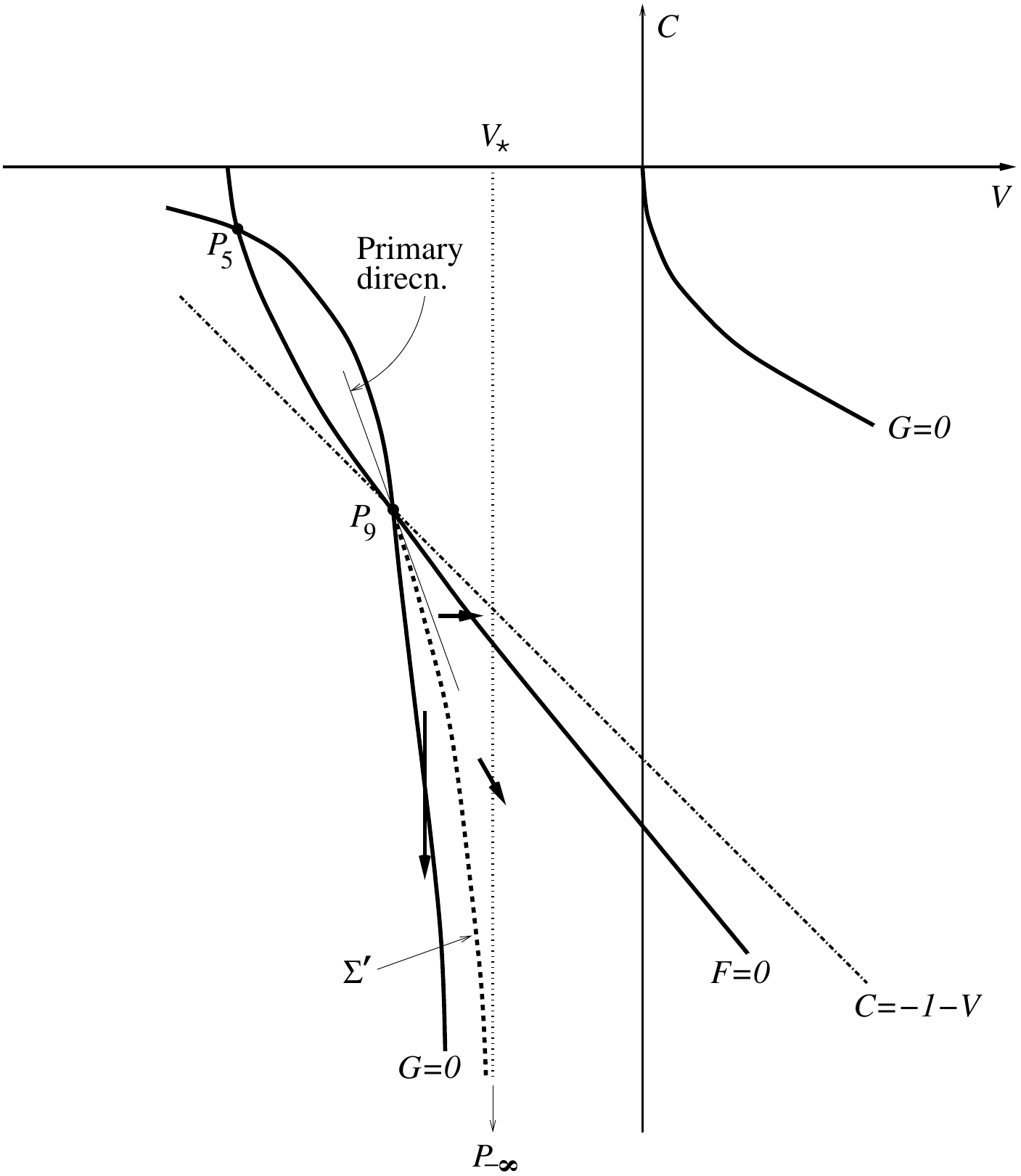}
	\caption{Schematic picture of the relative locations of the zero-level sets 
	of $F$ and $G$ (solid thick curves), the primary direction (solid thin line),
	and the critical line $L_-=\{C=-1-V\}$ (dash-dot) near $P_9$. Also shown is the 
	vertical asymptote $V=V_*$ of the zero-level of $G$ (thin dotted), as well as the 
	unique $P_9P_{-\infty}$-trajectory $\Sigma'$ (thick dotted). The three thick arrows 
	indicate the direction of flow for solutions to the original ODE system 
	\eq{V_sim2}-\eq{C_sim2} as $x>0$ increases.}\label{Figure_1}
\end{figure} 
%%%%%%%%%%%%%%%%%%%%%%%%%%%%%%%%%%%%%%%%%%%%

%%%%%%%%%%%%%%%%%%%%%%%%%%%%%%%%%%%%%%%%
\subsection{The node-saddle connection $P_9P_{-\infty}$}\label{9-infty_connecn}
%%%%%%%%%%%%%%%%%%%%%%%%%%%%%%%%%%%%%%%%
Referring again to Figure  \ref{Figure_1}, let $\Omega$ be the unbounded open 
region in the 3rd quadrant of the $(V,C)$-plane which is bounded by the curves 
$\{V=V_*\}$, $\{G=0\}$, and $\{F=0\}$, i.e.,
\[\Omega=\{(V,C)\,|\, V<V_*,\, C<C_9,\, G(V,C)>0,\, F(V,C)<0\}.\]
It is then routine to use two properties (p1)-(p2) above
to verify that trajectories inside $\Omega$ of the original similarity ODEs \eq{V_sim2}-\eq{C_sim2},
move in a south-east manner as $x$ increases. (Recall that $x>0$ along such solutions,
cf.\ the analysis in Section \ref{P1_P3}.) Furthermore, $\Omega$ contains no critical points and,
as indicated by thick arrows in Figure  \ref{Figure_1}, 
any trajectory that meets $\partial\Omega\smallsetminus\{P_9\}$ must exit $\Omega$.
Finally, according to Proposition \ref{rel_prop}, $P_{-\infty}$ is a saddle point and $P_9$
is a node. It follows by continuity that among the infinitely many trajectories entering
$\Omega$ at $P_9$, there is a unique trajectory, denoted $\Gamma'$, which connects
$P_9$ to $P_{-\infty}$.

%%%%%%%%%%%%%%%%%%%%%%%%%%%%%%%%%%%%%%%%
\subsection{The critical points $P_4$ and $P_5$}\label{P_4_P_5}
%%%%%%%%%%%%%%%%%%%%%%%%%%%%%%%%%%%%%%%%
The locations of the critical points $P_4$ and $P_5$
are given by Lemma \ref{pres_locn}. To determine their types, we make use 
of the following expressions for the Wronskian at $P_4$ and $P_5$
(cf.\ \eq{W_9_eqn}; for proofs see p.\ 323 in \cite{laz}):
\beq\label{W_5}
	W_4=W_5=W_5(n,\gamma,\lambda)=KC_5^2(V_5-V_7)(V_5-V_9),
\eeq
%and
%\beq\label{W_7}
%	W_7=W_7(n,\gamma,\lambda)=KC_7^2(V_7-V_5)(V_7-V_9),
%\eeq
where the positive constant $K$ is given in \eq{K}. Recall that $V_5=V_4$, 
$V_7=V_-$, and $V_9=V_+$. Therefore, by combining Lemma \ref{pres_locn}
and Lemma \ref{R^2_pos_suff}, we obtain that
$W_4=W_5<0$, so that $P_4$ and $P_5$ are saddle points whenever
$(\gamma,\lambda)$ is relevant.

Summarizing the analysis in Sections \ref{presence_type} and \ref{further_props}, 
we have:
%%%%%%%%%%%%%%%%%%%%%%%%%%%%%%%%%%
\begin{proposition}\label{locn_of_crit_points}
Assume $n=2$ or $n=3$, $\gamma>1$, $\kappa=\bar\kappa(\gamma,\lambda)$,
and $\lambda>1$ is sufficiently close to $1$. Then:
\begin{itemize}
	\item All of the critical points $P_4$-$P_9$ are present;
	\item $P_8$ and $P_9$ are nodes for \eq{CV_ode};
	\item $P_7$ and $P_9$ are located along $L_-$ within the half-strip $\{(V,C)\,|\,-1<V<V_*,\, C<0\}$;
	\item $P_6$ and $P_8$ are located along $L_+$ within the half-strip $\{(V,C)\,|\,-1<V<V_*,\, C>0\}$;
	\item $P_4$ and $P_5$ are saddles for \eq{CV_ode};
	\item $P_5$ is located within the set $\{(V,C)\,|\,V_-<V<V_+,\, -(1+V)<C<0\}$;
	\item $P_4$ is located within the set $\{(V,C)\,|\,V_-<V<V_+,\, 0<C<1+V\}$;
	\item There is a unique trajectory $\Gamma'$ of the original similarity ODEs 
	\eq{V_sim2}-\eq{C_sim2} which joins $P_9$ to $P_{-\infty}$.
\end{itemize}
%%%%%%%%%%%%%%%%%%%%%%%%%%%%%%%%%%
\end{proposition}

%%%%%%%%%%%%%%%%%%%%%%%%%%%%%%%%%%
%%%%%%%%%%%%%%%%%%%%%%%%%%%%%%%%%%
\section{Blowup solutions continuous away from the point of collapse}\label{global_flow}
%%%%%%%%%%%%%%%%%%%%%%%%%%%%%%%%%%
%%%%%%%%%%%%%%%%%%%%%%%%%%%%%%%%%%
We proceed to argue for the statements in Theorem \ref{thm}.
The main issue is to establish that, for
$n=3$, $\gamma>1$, $\kappa=\bar\kappa(\gamma,\lambda)$, and with $\lambda>1$ 
sufficiently close to $1$, the particular trajectory of \eq{V_sim2}-\eq{C_sim2} which 
passes {\em vertically} through the origin $P_1$ 
in the $(V,C)$-plane, connects the nodes $P_8$ and $P_9$. Thanks to the symmetries 
in \eq{symms}, it suffices to show that its part in the lower half-plane, denoted $\Sigma$, 
connects $P_1$ to $P_9$.

To argue for this, we shall shall show in Section \ref{Sigma} below that, under the stated conditions on $n$, 
$\gamma$, $\kappa$, and $\lambda$, the trajectory $\Sigma$ is located {\em strictly}
between two parabolae (or, rather, parts of parabolae)
\[\Pi_1=\{V=-\beta_1 C^2,\, C_5<C<0\},\qquad 
\Pi_2=\{V=-\beta_2 C^2,\, C_9<C<0\},\] 
in the third quadrant and which satisfy
\beq\label{parabolae_props}
	\beta_1>\beta_2,\qquad P_5\in\Pi_1,\qquad\text{and}\qquad P_9\in\Pi_2.
\eeq
Let $\mathcal R$ denote the region in the 3rd quadrant which is bounded by $\Pi_1$, $\Pi_2$, 
and that part of $\{G=0\}$ which is located between $P_5$ and $P_9$; see Figure \ref{Figure_2}.
It follows from the expressions for $F$, $G$, and $D$ that trajectories of \eq{V_sim2}-\eq{C_sim2}
move in a south-west manner within $\mathcal R$ as $x>0$ increases. 

The argument in Section \ref{Sigma} will establish the following properties: 
\begin{itemize}
	\item[(A)] $\Sigma$ starts out from $P_1$ strictly between $\Pi_1$ and $\Pi_2$;
	\item[(B)] Any trajectory of \eq{V_sim2}-\eq{C_sim2} which meets $\Pi_1$ moves into $\mathcal R$; and
	\item[(C)] Any trajectory of \eq{V_sim2}-\eq{C_sim2} which meets $\Pi_2$ moves into $\mathcal R$.
\end{itemize}
Taking these properties for granted for now, it follows that $\Sigma$ does not exit $\mathcal R$ 
along $\Pi_1$ or $\Pi_2$. It therefore reaches $\{G=0\}$ (vertically) at a point located {\em strictly} 
between $P_5$ and $P_9$, and enters the ``eye-shaped'' region 
\beq\label{eye}
	\mathcal E:=\{(V,C)\,|\, V_5<V<V_9,\, G(V,C)<0<F(V,C)\}.
\eeq
%%%%%%%%%%%%%%%%%%%%%%%%%%%%%%%%%%%%%%%%%%%%
%%%%%%%%%%%%%%%%%%%
%	FIGURE 
%%%%%%%%%%%%%%%%%%%
\begin{figure}
	\centering
	\includegraphics[width=11cm,height=11cm]{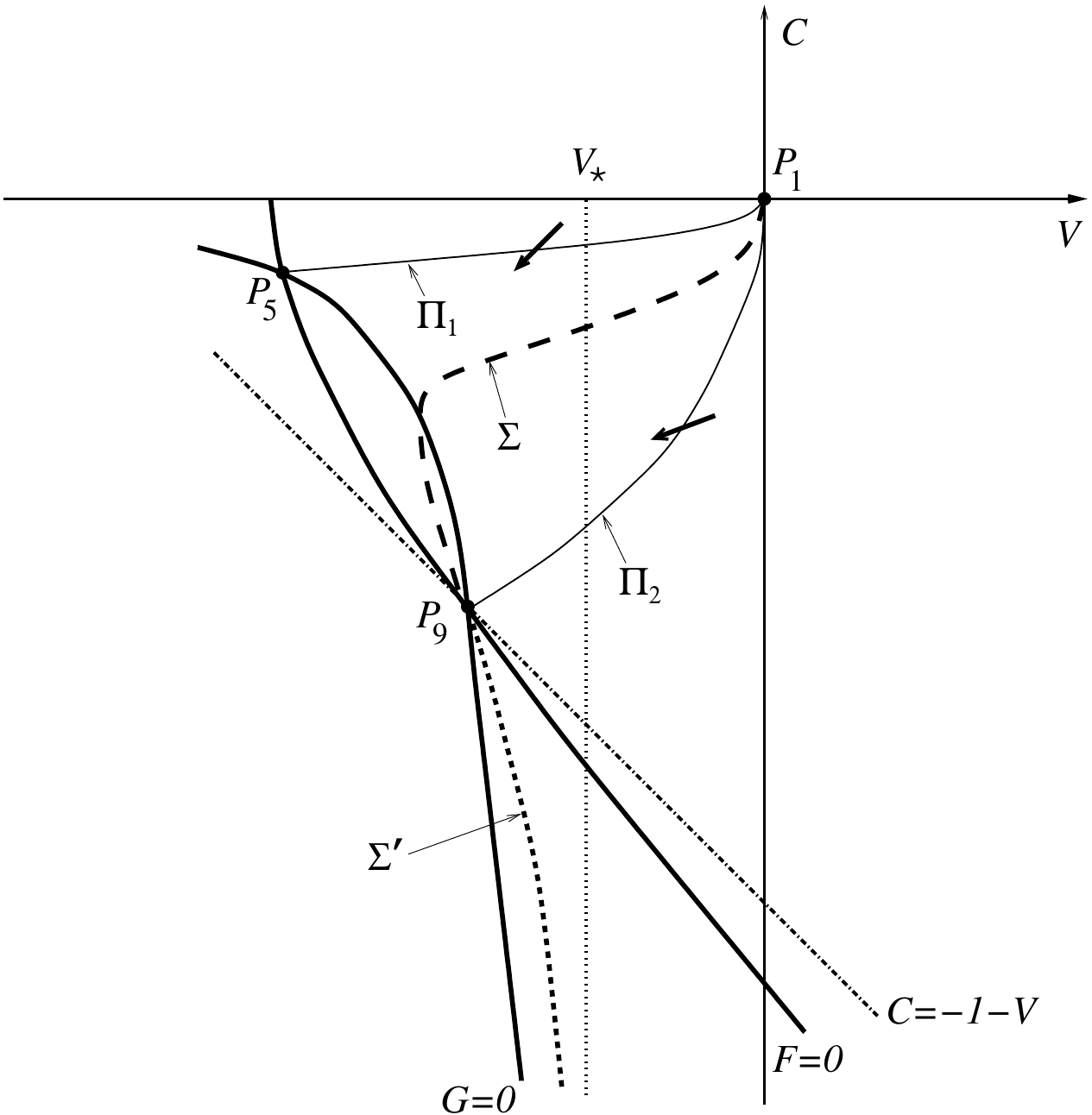}
	\caption{Schematic picture showing the $\Sigma$ trajectory (thick long dashes) which leaves the
	origin $P_1$ vertically and connects $P_1$ to $P_9$. $\Sigma$ is bounded above and below by
	the parabolic barriers $\Pi_1$ and $\Pi_2$ (thin curves), respectively, before entering the eye-shaped region
	between $P_5$ and $P_9$. The two thick arrows 
	indicate how solutions to the original ODE system 
	\eq{V_sim2}-\eq{C_sim2} cross the barriers as $x>0$ increases.}\label{Figure_2}
\end{figure} 
%%%%%%%%%%%%%%%%%%%%%%%%%%%%%%%%%%%%%%%%%%%%
Recalling that $x>0$ along solutions of the similarity ODEs \eq{V_sim2}-\eq{C_sim2}
in the lower half-plane,
it is routine to verify that $\mathcal E$ is a trapping region for the system 
\eq{V_sim2}-\eq{C_sim2}. Any solution of \eq{V_sim2}-\eq{C_sim2} which enters 
$\mathcal E$ must therefore approach the node at $P_9$ as $x$ increases.
Thus, modulo the properties (A), (B), (C), we have that $\Sigma$ connects $P_1$ to $P_9$.

Having reached $P_9$, the trajectory $\Sigma$ is continued as the unique 
node-saddle trajectory $\Gamma'$ connecting $P_9$ to $P_{-\infty}$ 
(cf.\ Section \ref{9-infty_connecn}); see Figure \ref{Figure_2}.

Next, let $\Sigma_\circ$ and $\Sigma'_\circ$ denote the reflections 
of $\Sigma$ and $\Sigma'$, respectively, about the $V$-axis.
With these we define the trajectory
\beq\label{complete_traj}
	\Gamma:=\Sigma'_\circ\cup\Sigma_\circ\cup\Sigma\cup\Sigma'.
\eeq
Recall from Section \ref{P_infs} that approaching 
$P_{+\infty}$ along $\Sigma'_\circ$ corresponds to $x\to-\infty$,
and that approaching $P_{-\infty}$ along $\Sigma'$ corresponds to $x\to+\infty$.
It follows that $\Gamma$ provides a trajectory connecting $P_{+\infty}$ to $P_{-\infty}$
via $P_8$, $P_1$, and $P_9$, and that the corresponding solution $(V(x),C(x))$ of the 
similarity ODEs \eq{V_sim2}-\eq{C_sim2} is defined and {\em continuous}
for all $x\in(-\infty,\infty)$. With $\Gamma$ so constructed we obtain the velocity and 
sound speed of the corresponding Euler flow from \eq{sim_vars}${}_{3,4}$.
%%%%%%%%%%%%%%%%%%%%%%%%%%%%%%%%%%%%%%%%%%%%
\begin{remark}
	We note that, strictly speaking, we have the freedom of
	choosing any positive $x$-value $x_9>0$, say, and then fixing the solution $(V(x),C(x))$ that 
	traverses $\Gamma$ by insisting that $(V(x_9),C(x_9))=P_9$.
	Choosing different values for $x_9$ will give physically distinct Euler flows; however,
	these are all trivially related via space-time scalings.
\end{remark}
%%%%%%%%%%%%%%%%%%%%%%%%%%%%%%%%%%%%%%%%%%%%

To complete the construction of the sought-for solution we need to specify the 
density field. For this we choose any positive constant for the right-hand side of \eq{CR},
fixing the function $R(x)$ along the solution under consideration, and define the 
associated flow variables $(\rho(t,r),u(t,r),c(t,r))$ according to \eq{sim_vars}.
This will complete the construction of a globally defined, radial Euler flow
which, by construction is continuous everywhere except at $(t,r)=(0,0)$, 
where it suffers amplitude blowup (in $\rho$, $u$, $c$, $p$, and $\theta$).
Modulo the properties (A), (B), (C), this establishes parts (i), (ii), and (iii) of Theorem \ref{thm}.
The next section provides the details of proving (A), (B), (C), and part (iv) of Theorem \ref{thm}
is established in Section \ref{part_iv}.

%%%%%%%%%%%%%%%%%%%%%%%%%%%%%%%%%%%%%%
\subsection{Analysis of the $\Sigma$ trajectory}\label{Sigma}
%%%%%%%%%%%%%%%%%%%%%%%%%%%%%%%%%%%%%%
To carry out the analysis outlined above, we first calculate the leading (quadratic) 
order behavior of $\Sigma$ near the origin $P_1$. That is, we determine $\beta>0$
such that 
\[\Sigma=\{V=-\beta C^2+O(C^3)\text{ for $C\lessapprox0$}\}.\]
Until further notice it is assumed that $n=2$ or $n=3$, $\gamma>1$, $\kappa=\bar\kappa$, and $\lambda>1$.
We  consider $V$ as a function of $C$ along $\Sigma$ and use \eq{CV_ode},
in the form $\frac{dV}{dC}=\frac{G(V,C)}{F(V,C)}$, to calculate
\beq\label{2nd_der}
	\textstyle\frac{d^2V}{d^2C}=\frac{G(FG_V-GF_V)+F(FG_C-GF_C)}{F^3}.
\eeq
From \eq{G}-\eq{F}, and with $V\approx-\beta C^2$, we have to leading order along 
$\Sigma$ near $P_1$ that
\begin{align*}
	G&=nC^2(V-V_*)-V(1+V)(\lambda+V)\approx(\beta\lambda-nV_*)C^2\\
	G_V&=nC^2-(1+V)(\lambda+V)-V(\lambda+V)-V(1+V)\approx -\lambda\\
	G_C&=2nC(V-V_*)\approx -2nV_*C\\
	F&=C\big[C^2-k_1(1+V)^2+k_2(1+V)-k_3\big]\approx-\lambda C\\
	F_V&=C[k_2-2k_1(1+V)]\approx(k_2-2k_1)C\\
	F_C&=3C^2-k_1(1+V)^2+k_2(1+V)-k_3\approx -\lambda,
\end{align*}
where we have used \eq{sum_k} and the fact 
that $\alpha=0$ since $\kappa=\bar\kappa$; also, $V_*$ is given by \eq{isntr_V_star}.
Using these expressions in \eq{2nd_der} gives
\[-2\beta=\textstyle\frac{d^2V}{d^2C}\big|_{C=0}=\frac{2nV_*}{\lambda},\]
or
\beq\label{beta}
	\beta=-\textstyle\frac{nV_*}{\lambda}=\frac{2\mu}{\eps\lambda}.
\eeq
Next, the parabola
\[\Pi_1=\{V=-\beta_1 C^2,\, C_5<C<0\}\]
is defined by the requirement that it passes through $P_5$, giving
\beq\label{beta_1}
	\beta_1=-\textstyle\frac{V_5}{C_5^2}=\frac{n(V_*-V_5)}{(1+V_5)(\lambda+V_5)}
	=\frac{2(2+n\eps)}{n\lambda\eps^2},
\eeq
where we have used \eq{g} and \eq{V_4_C_4} (recall that $V_5=V_4$). Similarly, 
the parabola
\[\Pi_2=\{V=-\beta_2 C^2,\, C_9<C<0\}\]
is defined by the requirement that it passes through $P_9$, giving
\beq\label{beta_2}
	\beta_2=-\textstyle\frac{V_9}{C_9^2}=-\frac{V_+}{(1+V_+)^2},
\eeq
where we have used that $P_9\in L_-$, and where $V_9=V_+$ is given by \eq{isntr_V_pm}-\eq{aQ}.
Note that it follows from Proposition \ref{locn_of_crit_points} that $P_5$ is located 
north-west of $P_9$ within the triangle 
\beq\label{triangle}
	\mathcal T:=\{(V,C)\,|\, -1< -1-V<C<0 \}.
\eeq
in the third quadrant (cf.\ Figure \ref{Figure_2}).

As argued above, in order to show that $\Sigma$ meets $\{G=0\}$ between $P_5$ and $P_9$, 
it suffices to establish properties (A), (B), (C), which amount to the following:
\begin{itemize}
	\item[(A)] $\Sigma$ starts out from $P_1$ between $\Pi_1$ and $\Pi_2$, i.e.,
	\beq\label{betas}
		\beta_1>\beta>\beta_2;
	\eeq
	\item[(B)] Whenever a trajectory $(V(C),C)$ of \eq{CV_ode} crosses $\Pi_1$ at a point $(-\beta_1C^2,C)$,
	with $C_5<C<0$, we have
	\beq\label{dV/dC_1}
		\textstyle\frac{dV}{dC}\big|_{\Pi_1}=\frac{G(-\beta_1C^2,C)}{F(-\beta_1C^2,C)}<-2\beta_1 C;
	\eeq
	this guarantees that solutions to \eq{V_sim2}-\eq{C_sim2} cross
	$\Pi_1$ into the region $\mathcal R$ as $x$ increases; 
	\item[(C)] Whenever a trajectory $(V(C),C)$ of \eq{CV_ode} crosses $\Pi_2$ at a point $(-\beta_2C^2,C)$,
	with $C_9<C<0$, we have
	\beq\label{dV/dC_2}
		\textstyle\frac{dV}{dC}\big|_{\Pi_2}=\frac{G(-\beta_2C^2,C)}{F(-\beta_2C^2,C)}>-2\beta_2 C;
	\eeq
	this guarantees that solutions to \eq{V_sim2}-\eq{C_sim2} cross
	$\Pi_2$  into the region $\mathcal R$ as $x$ increases.
\end{itemize}
We proceed to verify that (A), (B), and (C) are all satisfied for $n=3$, $\gamma>1$,
$\kappa=\bar\kappa$, provided $\lambda>1$ is sufficiently close to $1$ (depending on $\gamma$). 
We note that it is only in verifying property (A) that we need to restrict to $n=3$.

%%%%%%%%%%%%%%%%%%%%%%%%%%
\subsubsection{Property (A)}\label{prop_A}
%%%%%%%%%%%%%%%%%%%%%%%%%%
Substituting from \eq{beta} and \eq{beta_1}, and rearranging, we have that $\beta_1>\beta$
amounts to $2+n\eps>\mu n\eps$, which is trivially satisfied whenever $\mu=\lambda-1\gtrapprox0$.

The analysis of the second inequality $\beta>\beta_2$ is more involved, and, as we shall see,
its validity requires $n=3$. Considering $V_*$ and $V_+$ as functions of $\mu$,
we have 
\beq\label{V_star_V_plus}
	V_*=-\textstyle\frac{2\mu}{n\eps}\qquad\text{and}\qquad 
	V_+=-\textstyle\frac{2\mu}{n\eps}+O(\mu^2),
\eeq
where we have used \eq{isntr_V_star} and \eq{isntr_V_pm}-\eq{aQ}. Using these 
to Taylor expand $\beta_2$ in \eq{beta_2} gives
\beq\label{beta_2_expn}
	\beta_2=\textstyle\frac{2}{m\eps}\mu+\textstyle\frac{2(\eps+4)}{m^2\eps^2}\mu^2+O(\mu^3).
\eeq
Also, according to \eq{beta}, we have 
\[\beta=\textstyle\frac{2}{\eps}\mu-\textstyle\frac{2}{\eps}\mu^2+O(\mu^3),\]
It follows from these expressions that the inequality $\beta>\beta_2$ fails for $n=2$ ($m=1$),
while it is satisfied for $n=3$ ($m=2$), when $\mu\gtrapprox0$ (depending on $\eps$, i.e., on $\gamma$).
Consequently, in what follows we restrict attention to $n=3$.

%%%%%%%%%%%%%%%%%%%%%%%%%%
\subsubsection{Property (B)}\label{prop_B}
%%%%%%%%%%%%%%%%%%%%%%%%%%
With $n=3$, $\eps=\gamma-1>0$, and $\kappa=\bar\kappa$ we have
\beq\label{beta_1_V_star_V_5_C_5_n=3}
	\beta_1=\textstyle\frac{2(3\eps+2)}{3\lambda\eps^2},\qquad
	V_*=-\textstyle\frac{2\mu}{3\eps},\qquad
	V_5=V_4=-\textstyle\frac{2\lambda}{3\eps+2},\qquad
	C_5^2=\textstyle\frac{3\lambda^2\eps^2}{(3\eps+2)^2}
\eeq
where we have used \eq{beta_1}, \eq{isntr_V_star}, \eq{V_4_C_4}, and \eq{g}.
Also, from \eq{ks}, we have
\beq\label{ks_n=3}
	k_1=\eps+1,\qquad
	k_2=\textstyle\frac{(2+\mu)\eps-2\mu}{2},\qquad\text{and}\qquad
	k_3=\textstyle\frac{\mu\eps}{2}.
\eeq
As $F>0$ along $\Pi_1$, we have that \eq{dV/dC_1} holds if and only if
\beq\label{aux_1}
	G(-\beta_1C^2,C)<-2\beta_1 CF(-\beta_1C^2,C)\qquad\text{for $C_5<C<0$.}
\eeq
Using the expressions for $F$ and $G$ from \eq{G}-\eq{F}, together with \eq{beta_1_V_star_V_5_C_5_n=3}-\eq{ks_n=3},
rearranging and dividing through by $\beta_1C^2>0$, and finally setting $Z=C^2$, we obtain the equivalent inequality
\beq\label{dV/dC_1_2}
	(2\eps+1)\beta_1^2 Z^2+\left\{1+[(\mu-2)\eps-(\mu+2)]\beta_1\right\}Z
	+(\lambda-\textstyle\frac{2\mu}{\beta_1\eps})<0\qquad\text{for $0<Z<Z_5:=C_5^2$.}
\eeq
We proceed to argue that \eq{dV/dC_1_2} is satisfied whenever $\lambda\gtrapprox1$ (i.e., $\mu\gtrapprox0$).
Denoting the left-hand side of \eq{dV/dC_1_2} by $\phi(Z)$, we have
\[\phi'(Z)=2(2\eps+1)\beta_1^2 Z+1+[(\mu-2)\eps-(\mu+2)]\beta_1.\]
Direct evaluations, using \eq{beta_1_V_star_V_5_C_5_n=3}, then give
\begin{itemize}
	\item $\phi(Z_5)=0$ for any $\lambda>1$ (note that this reflects the fact that \eq{aux_1}, and hence
	\eq{dV/dC_1_2}, is satisfied with equality for $Z=Z_5$ since $F(P_5)=G(P_5)=0$); and
	\item $\phi'(Z_5)=-\textstyle\frac{9\eps+4}{3\eps}+O(\mu)<0$ for $\lambda\gtrapprox1$.
\end{itemize}
As $\phi(Z)$ is a quadratic polynomial in $Z$ with a positive leading coefficient,
it follows that $\phi(Z)>0$ for $0<Z<Z_5$, establishing \eq{dV/dC_1_2}.
This shows that Property (B) holds whenever $\lambda>1$ is 
sufficiently close to $1$ (depending on $\gamma$).

%%%%%%%%%%%%%%%%%%%%%%%%%%
\subsubsection{Property (C)}\label{prop_C}
%%%%%%%%%%%%%%%%%%%%%%%%%%
We assume $n=3$, $\eps=\gamma-1>0$, and $\kappa=\bar\kappa$.
%we use  \eq{isntr_V_pm}-\eq{aQ} to Taylor expand $V_+$ about $\mu=0$ to get
%$V_+=-\frac{\mu}{\eps}+O(\mu^2)$. In turn this gives
%\beq\label{C_9_n=3}
%	C_9^2=(1+V_+)^2=1-\textstyle\frac{2\mu}{\eps}+O(\mu^2)
%\eeq
As $F>0$ along $\Pi_2$, we have that \eq{dV/dC_2} holds if and only if
\beq\label{aux_2}
	-2\beta_2 CF(-\beta_2C^2,C)<G(-\beta_2C^2,C)\qquad\text{for $C_9<C<0$.}
\eeq
Using the expressions for $F$ and $G$ from \eq{G}-\eq{F} together with \eq{ks_n=3},
rearranging and dividing through by $\beta_2C^2>0$, and finally setting $Z=C^2$, 
we obtain the equivalent inequality
\beq\label{dV/dC_2_2}
	(2\eps+1)\beta_2^2 Z^2+\left\{1+[(\mu-2)\eps-(\mu+2)]\beta_2\right\}Z
	+(\lambda-\textstyle\frac{2\mu}{\beta_2\eps})<0,\qquad\text{for $0<Z<Z_9:=C_9^2$.}
\eeq
We proceed to argue that \eq{dV/dC_2_2} is satisfied whenever $\lambda\gtrapprox1$.
Denoting the left-hand side of \eq{dV/dC_2_2} by $\psi(Z)$, 
%we have
%\[\psi'(Z)=2(2\eps+1)\beta_2^2 Z+1+[(\mu-2)\eps-(\mu+2)]\beta_2.\]
direct evaluations, using \eq{beta_2_expn} give
\begin{itemize}
	\item $\psi(0)=-1+O(\mu)<0$ for $\lambda\gtrapprox1$;
	\item $\psi(Z_9)=0$ for any $\lambda>1$ (note that this reflects the fact that \eq{aux_2}, and hence
	\eq{dV/dC_2_2}, is satisfied with equality for $Z=Z_9$ since $F(P_9)=G(P_9)=0$).
%	\item $\psi'(Z_9)=1+O(\mu)>0$ for $\lambda\gtrapprox1$.
\end{itemize}
As $\psi(Z)$ is a quadratic polynomial in $Z$ with a positive leading coefficient,
it follows that $\psi(Z)<0$ for $0<Z<Z_9$, establishing \eq{dV/dC_2_2}.
This shows that Property (C) holds whenever $\lambda>1$ is sufficiently close 
to $1$ (depending on $\gamma$).

With this we have argued for parts (i), (ii), and (iii) of Theorem \ref{thm}.

%%%%%%%%%%%%%%%%%%%%%%%%%%%%%%%%%%%%%%
\subsection{Proof of part (iv) of Theorem \ref{thm}}\label{part_iv}
%%%%%%%%%%%%%%%%%%%%%%%%%%%%%%%%%%%%%%
With the existence of the continuous solution trajectory $\Gamma$ in \eq{complete_traj} established, 
part (iv) of Theorem \ref{thm} is argued for as follows. 

Recall from the construction above that the trajectory $\Sigma_{89}:=\Sigma\cup\Sigma_\circ$ 
is the unique $P_8P_9$-connection that passes vertically through the origin $P_1$.
For the parameter regime described in Theorem \ref{thm}, it was established above that $\Sigma_{89}$ 
meets $\{G=0\}$ at a point strictly between $P_5$ and $P_9$ in the lower half-plane.
As $\Sigma_\circ$ is the reflection of $\Sigma$ about the $V$-axis, it follows that $\Sigma_{89}$
also meets $\{G=0\}$ at a symmetrically located point located strictly between $P_4$ and $P_8$
in the upper half-plane. It now follows by continuity that there are nearby trajectories connecting the node at 
$P_8$ to the node at $P_9$ via $P_1$. Any such ``perturbed'' $P_8P_9$ trajectory $\tilde\Gamma_{89}$
passes through $P_1$ with a (large) strictly positive or (large) strictly negative slope. We note that 
$\tilde\Gamma_{89}$ must arrive and leave $P_1$ with the same slope; see Section \ref{unphys_kink_soln} below.
With the same notation as in \eq{complete_traj} we therefore have that 
\beq\label{pertrbd_compl_traj}
	\tilde\Gamma:=\Sigma'_\circ\cup\tilde \Gamma_{89}\cup\Sigma'.
\eeq
is the trajectory of a solution $(\tilde V(x),\tilde C(x))$ to the 
similarity ODEs \eq{V_sim2}-\eq{C_sim2} which is defined and continuous
for all $x\in(-\infty,\infty)$. 

Finally, if $\tilde\Gamma_{89}$ passes through $P_1$ with a positive slope $\ell>0$, then 
we are in Case 1 described in Section \ref{P1_P3}, and the fluid is  everywhere flowing away 
from the center of motion at time $t=0$. Similarly, if $\tilde\Gamma_{89}$ passes through 
$P_1$ with a negative slope $\ell<0$, Case 2 in Section \ref{P1_P3} applies, and the fluid 
is everywhere flowing toward the center of motion at time $t=0$. 

This concludes the proof of Theorem \ref{thm}.

%%%%%%%%%%%%%%%%%%%%%%%%%%%%%%%%%%%%%%%%%%%%
\section{Additional remarks}\label{remarks}
%%%%%%%%%%%%%%%%%%%%%%%%%%%%%%%%%%%%%%%%%%%%
\subsection{Blowup with a converging or diverging velocity field}\label{remarks_1}
%%%%%%%%%%%%%%%%%%%%%%%%%%%%%%%%%%%%%%%%%%%%
As noted earlier, since $\lambda>1$ and $\bar\kappa<0$, all the solutions constructed above
suffer blowup of density and pressure at the center of motion at time of collapse. 
Also, in Case 2 above the fluid is everywhere moving toward the origin $r=0$ at time of collapse
(cf.\ discussion in Section \ref{P1_P3}).
It is surprising that, under these circumstances, no shock wave appears in 
the flows under consideration.
We interpret this as demonstrating how compressibility, which 
is typically seen as a necessary ingredient in shock formation, can play a dual role
in preventing shocks from developing. 

Next, consider Case 1: Again the density blows up at $r=0$ at time of collapse, but this 
now occurs while all fluid particles are moving away from the origin. It might seem 
that the latter property should tend to dilute the fluid, and thus prevent blowup of the density
(recall that the density is bounded at any time strictly prior to collapse).
However, what determines the rate of change of density along particle trajectories
(i.e., $\dot\rho=\rho_t+u\rho_r=-\rho u_r$),  is the sign of the velocity gradient $u_r$. 
For the self-similar flows under consideration
\[\dot\rho =-\textstyle\frac{1}{\lambda}r^{\bar\kappa-\lambda}R(x)
\left(\lambda V'(x)-\textstyle\frac{V(x)}{x}\right).\]
In Case 1 we have $\frac{V(x)}{x}\to V'(0)<0$ as $x\to 0$, so that 
\[\dot\rho(0,r)=-\textstyle\frac{1}{\lambda}r^{\bar\kappa-\lambda}R(0)(\lambda-1)V'(0)>0.\]
This shows that all fluid particles undergo an increase in density at time $t=0$ in this case,
and that this increase is more pronounced the closer the particle is to the center of motion.

%%%%%%%%%%%%%%%%%%%%%%%%%%%%%%%%%%%%%%%%%%%%
\subsection{On uniqueness}\label{remarks_2}
%%%%%%%%%%%%%%%%%%%%%%%%%%%%%%%%%%%%%%%%%%%%
As indicated in Section \ref{intro}, it is reasonable to ask if the strong 
singularities displayed by the type of solutions considered above could 
yield examples of non-unique continuation. 

In this section we consider two possible scenarios for non-uniqueness.
While one is easily dispelled with, the other requires more analysis to evaluate.

%%%%%%%%%%%%%%%%%%%%%%%%%%%%%%%%%%%%%%%%%%%%
\subsubsection{Passing through $P_1$ with a kink}\label{unphys_kink_soln}
%%%%%%%%%%%%%%%%%%%%%%%%%%%%%%%%%%%%%%%%%%%%
Recall from Section \ref{P1_P3} that  the origin $P_1$ is a star point 
for the reduced  similarity ODE \eq{CV_ode}, with trajectories approaching with any slope 
$\ell\in[-\infty,+\infty]$. Consider the trajectory $\Gamma$ which was built in 
Section \ref{global_flow}; its $P_8P_9$-part (i.e., $\Sigma_\circ\cup\Sigma$) was 
constructed specifically to pass vertically through the origin $P_1$. Having reached
$P_1$ from $P_8$ along $\Sigma_\circ$, we could try to leave $P_1$ along a 
trajectory $\tilde\Sigma$ which is a slight perturbation of $\Sigma$ - slight enough 
that $\tilde\Sigma$ still connects $P_1$ to $P_9$. With notation as in \eq{complete_traj},
we would thus obtain a complete trajectory 
$\tilde\Gamma:=\Sigma'_\circ\cup\Sigma_\circ\cup\tilde\Sigma\cup\Sigma'$, 
which connects $P_{+\infty}$ to $P_{-\infty}$, via $P_8$ and $P_9$, and 
passes through $P_1$. This would appear to generate an Euler flow which agrees
for $t<0$ with the one corresponding to the trajectory $\Gamma$, but it would differ from 
it for $t>0$.

However, in contrast to $\Gamma$, the trajectory $\tilde\Gamma$ passes through $P_1$ with a kink, 
and this renders the corresponding Euler flow unphysical.
Indeed, since $\tilde\Gamma$ approaches $P_1$ with different slopes 
as $x\to 0\pm$, \eq{well_bhvd} shows that at least one  
of $\frac{V(x)}{x}$ or $\frac{C(x)}{x}$ must suffer a jump as $x$ passes through zero.
According to \eq{sim_vars} this gives an Euler flow which suffers an unphysical 
jump in velocity or in sound speed (and thus in density) at all locations $r>0$ at time $t=0$.

%%%%%%%%%%%%%%%%%%%%%%%%%%%%%%%%%%%%%%%%%%%%
\subsubsection{Non-uniqueness via an outgoing shock?}
%%%%%%%%%%%%%%%%%%%%%%%%%%%%%%%%%%%%%%%%%%%%
The second scenario is more involved. For this we consider using a solution 
of the type described by Theorem \ref{thm} up to time of collapse, but then inserting
an outgoing shock in the flow for $t>0$ and emanating from $(t,r)=(0,0)$. 

For a solution as described by Theorem \ref{thm}, 
consider the point $P(x)=(V(x),C(x))$ as it traverses the $P_1P_9$-trajectory 
$\Sigma$ for $0\leq x\leq x_9$. For each such $x$-value we calculate the point
$P_H(x)=(V_H(x),C_H(x))$ with the property that the physical states corresponding to 
$P(x)$ and $P_H(x)$ satisfy the Rankine-Hugoniot relations for an admissible 3-shock.
We refer to 
\[\Sigma_H:=\{(V_H(x),C_H(x))\,|\, 0< x< x_9\}\] 
as the Hugoniot locus of $\Sigma$. (It follows from the Rankine-Hugoniot conditions 
and admissibility that $\Sigma_H$ is a uniquely determined and differentiable curve
located below the critical line $L_-$; see \cites{laz,jls1}. We omit the details.)

If the curve $\Sigma_H$ intersects the $P_9P_{-\infty}$-trajectory 
$\Sigma'$ constructed earlier, for an $x$-value $x_s\in(0,x_9)$, then we would 
have an example of non-unique continuation for the Euler system. 
Specifically, one solution would be the earlier constructed flow described by 
Theorem \ref{thm} which moves along the trajectory $\Gamma$. 
The other solution would be constructed by using the part of $\Gamma$ corresponding to 
$x<x_s$, then jumping to the point $P_H(x_s)\in\Sigma'$, and finally moving along $\Sigma'$ as
$x$ increases from $x_s$ to $+\infty$. The two corresponding Euler 
flows would agree up to time of collapse,
but differ for $t>0$ as one contains a shock and the other does not.

To determine whether the Hugoniot locus $\Sigma_H$ can in fact intersect the trajectory 
$\Sigma'$ turns out to be somewhat subtle. What we {\em can} say is the following.
First, it follows from the Rankine-Hugoniot conditions that if $\mathcal C$ is a 
differentiable curve (a trajectory of \eq{CV_ode} or not) which tends to a point $\bar P$ on 
the critical line $L_-$, then its Hugoniot locus approaches the same point. 
A further analysis shows that if $\mathcal C$ approaches $\bar P$ with slope $\sigma$, 
then its Hugoniot locus tends to $\bar P$ with the slope 
\beq\label{sigma_H}
	\sigma_H(\sigma):=\textstyle\frac{\gamma-1}{2}+(\gamma+1)
	\textstyle\frac{\left(\sigma-\frac{\gamma-1}{2}\right)}{(\gamma-3-4\sigma)}.
\eeq
(In particular, if $\mathcal C$ is a trajectory of \eq{CV_ode} which approaches 
$L_-$ at a non-triple point, then \eq{non_obvious_reln} shows that it does so with 
slope $\sigma=\frac{\gamma-1}{2}$, and it follows from \eq{sigma_H} that $\sigma_H=\sigma$ in this case.)
Applying \eq{sigma_H} to $\mathcal C=\Sigma$, which approaches the triple point $\bar P=P_9\in L_-$
with the primary slope $L_1$ (cf.\ Lemma \ref{prim_sec_direcns}), we get that 
$\Sigma_H$ approaches $P_9$ with slope $\sigma_H(L_1)$. 
According to Lemma \ref{rel_locns_P9} we have $L_1<-1$ for the solutions under consideration,
which, according to \eq{sigma_H}, gives $\sigma_H(L_1)>-1$. This implies that, at least near $P_9$, the Hugoniot locus 
$\Sigma_H$ is located within $\{V<V_9\}$, and thus does {\em not} meet $\Sigma'$,
which is located within the strip $\{V_9<V<V_*\}$. 

This indicates that, in the isentropic setting under consideration, the 
flows described by Theorem \ref{thm} can {\em not} be ``tweaked'' to give 
examples of non-unique continuation for the compressible Euler system. However, 
we have not been able to settle this issue, e.g.\ by showing that the entire Hugoniot 
locus $\Sigma_H$ is located within $\{V<V_9\}$ (as indicated by numerical tests). 
It would be of interest to provide a conclusive answer; the recent detailed analysis in
\cite{jls1} might be of help with this.

%BIBLIOGRAPHY

\end{document}